\documentclass[11pt]{amsart}

\usepackage{amssymb, amsmath}
\usepackage{diagbox}

\usepackage{tikz, tikz-cd}
\usepackage{enumitem}
\usepackage{mathtools}
\usepackage{todonotes}
\usepackage{hyperref}
\hypersetup{
    colorlinks=true,
    citecolor=blue,
    linkcolor=blue,
    filecolor=magenta,      
    urlcolor=blue,
}

\usepackage[OT2,T1]{fontenc}
\DeclareSymbolFont{cyrletters}{OT2}{wncyr}{m}{n}
\DeclareMathSymbol{\Sha}{\mathalpha}{cyrletters}{"58}
\theoremstyle{plain}
 \newtheorem{thm}{Theorem}[section]
 
 \newtheorem{cor}[thm]{Corollary}
\newtheorem{lem}[thm]{Lemma}
\newtheorem{prop}[thm]{Proposition}

\theoremstyle{definition}

\theoremstyle{plain}

\newtheorem{lemma}[thm]{Lemma}

\newtheorem{proposition}[thm]{Proposition}
\theoremstyle{definition}

\newtheorem{claim}[thm]{Claim}

\newtheorem{defin}[thm]{Definition}

\newtheorem{remark}[thm]{Remark}

\numberwithin{equation}{section}

\newcommand{\sA}{{\mathcal A}}
\newcommand{\sB}{{\mathcal B}}
\newcommand{\sC}{{\mathcal C}}

\newcommand{\sE}{{\mathcal E}}

\newcommand{\sG}{{\mathcal G}}

\newcommand{\sJ}{{\mathcal J}}
\newcommand{\sK}{{\mathcal K}}

\newcommand{\sN}{{\mathcal N}}

\newcommand{\sR}{{\mathcal R}}
\newcommand{\sS}{{\mathcal S}}
\newcommand{\sT}{{\mathcal T}}

\newcommand{\Q}{{\mathbb Q}}

\newcommand{\NN}{\ensuremath{\mathbb{N}}}
\newcommand{\hol}{\ensuremath{\mathcal{O}}}

\newenvironment{dedication}
        {\begin{quotation}\begin{center}\begin{em}}
        {\par\end{em}\end{center}\end{quotation}}
        
\newcommand\om{\omega}
\newcommand\la{\lambda}
\newcommand\Lam{\Lambda}
\newcommand\s{\sigma}

\newcommand\e{\epsilon}
\newcommand\al{\alpha}
\newcommand\be{\beta}
\newcommand\Ga{\Gamma}
\newcommand\De{\Delta}
\newcommand\ga{\gamma}
\newcommand\de{\delta}

\DeclareMathOperator{\Pic}{Pic}

\DeclareMathOperator{\Hom}{Hom}

\DeclareMathOperator{\Tors}{Tors}

\newcommand{\HH}{\ensuremath{\mathbb{H}}}

\newcommand{\ra}{\ensuremath{\rightarrow}}

\newcommand{\CC}{\mathbb{C}}

\newcommand{\PP}{\mathbb{P}}
\newcommand{\QQ}{\mathbb{Q}}
\newcommand{\RR}{\mathbb{R}}
\newcommand{\ZZ}{\mathbb{Z}}

\newcommand{\Aut}{\mathrm{Aut}}

\newcommand{\I}{\mathrm{I}}

\newcommand{\id}{\mathrm{id}}

\newcommand{\Inn}{\mathrm{Inn}}

\newcommand{\Jac}{\mathrm{Jac}}

\newcommand{\MW}{\mathrm{MW}}
\newcommand{\NS}{\mathrm{NS}}

\newcommand{\rank}{\mathrm{rank\,}}

\newcommand{\tor}{\mathrm{tor}}

\begin{document}
\title[
 Numerically
trivial automorphisms III]{ 
 Numerically
 trivial automorphisms of  
   elliptic surfaces III: the  non-algebraic case }
\author{Fabrizio Catanese}
\author{Matthias Sch\"utt}
\date{\today}

\address{Lehrstuhl Mathematik VIII, Mathematisches Institut der Universit\"{a}t
Bayreuth, NW II, Universit\"{a}tsstr. 30,
95447 Bayreuth, and Accademia Nazionale dei Lincei, Palazzo Corsini, via della Lungara 10, 00165 Roma, Italia.}
\email{Fabrizio.Catanese@uni-bayreuth.de, Fabrizio.Catanese@gmail.com}


\address{Institut f\"ur Algebraische Geometrie, Leibniz Universit\"at
  Hannover, Welfengarten 1, 30167 Hannover, Germany\\ and\;\;\;\;\;\;\;\;\;\;\;\;\;\;\;\;\;\;\;\;\;\;\;\;\;\;\;\;\;\;\;\;\;\;\;\;\;
  \linebreak
  Riemann Center for Geometry and Physics, Leibniz Universit\"at
  Hannover, Appelstrasse 2, 30167 Hannover, Germany}
\email{schuett@math.uni-hannover.de}


\keywords{Compact K\"ahler manifolds, algebraic surfaces, elliptic surfaces, automorphisms, cohomologically trivial automorphisms, numerically trivial automorphisms, Enriques--Kodaira classification}

\subjclass[2010]{14J50, 14J80,  14J27, 14H30, 14F99, 32L05, 32M99, 32Q15, 32Q55}

\begin{abstract}
In this third part we  study the group $\Aut_{\QQ}(S)$ of  numerically trivial automorphisms of  a non-algebraic  properly elliptic surface $S$. We treat completely the K\"ahler case.
\end{abstract}
\maketitle

\begin{dedication}
This article is dedicated to Fedya Bogomolov on the occasion of his forthcoming 80th birthday,
with friendship and admiration.
\end{dedication}

\setcounter{tocdepth}{1}

\tableofcontents
\section{Introduction}

This  note extends the results of \cite{CFGLS24} and \cite{CLS24} concerning the group $\Aut_{\QQ}(S)$ 
of the numerically trivial automorphisms  (those which act trivially on rational cohomology) of a complex algebraic elliptic surface $S$ to the more general case
of a  
 non-algebraic elliptic surface $S$.

If $S$ is a non-algebraic compact 
smooth complex surfaces $S$, then, 
by  Kodaira's classification, we know  that there are two cases, distinguished by the algebraic
dimension $alg(S)$ of $S$, which is the transcendence degree over $\CC$ of the field $\CC(S)$
of meromorphic functions on $S$ (see Theorem 3.3 of \cite{BH}).

\begin{enumerate}
\item[(I)] \ $alg(S) =1 \Longleftrightarrow S \  \ {\rm is \ elliptic} $
\item[(0)]  \
$alg(S) =0 \Longleftrightarrow \CC(S) = \CC \Longleftrightarrow  S \  \ {\rm is \ not \ elliptic} .$
\end{enumerate} 

In this paper we shall deal  essentially  with case (I), where there is a natural
 elliptic fibration 
$$f \colon S \ra B$$ 
where $B$ is a compact complex  curve,
 $f$ has connected fibres, and we may  assume that $f$ is relatively minimal.
There is a further distinction (ibidem, table on page 415) according to the 
Kodaira dimension $Kod (S)$:

\begin{enumerate}
\item
$Kod(S) =1 \Leftrightarrow S \  \ {\rm is \ properly \ elliptic} $, this means that  the elliptic fibration is associated to the sections of $H^0(\hol_S(m K_S))$, 
for $m$ sufficiently large.

\item
$Kod (S) =0$, and $S$ is either a torus, or a K3 surface, or a Kodaira surface;
 here $K_S$ is trivial, unless we have  a secondary Kodaira surface, in which case $K_S$ is torsion.

\item
\label{case:Hopf}
$Kod (S) = -\infty \Leftrightarrow$
$S$ is an elliptic Hopf surface.
\end{enumerate} 

Indeed, Kodaira in his  paper   \cite[p.\ 790]{kodairastructure1}  divided minimal complex surfaces into seven classes,
among which the non-algebraic surfaces can only be K3 surfaces, or tori, or
surfaces in the classes $(IV_0), (VI_0), (VII_0)$:

\begin{enumerate}
\item[Class $(IV_0)$] 
consists of  elliptic surfaces with $b_1(S)$ even, $p_g(S)\geq 1, P_{12} (S) \geq 1$;
\item[Class $(VI_0)$]  
consists of elliptic surfaces 
with $b_1(S)$ odd, $P_{12} (S) \geq 1$; 
\item[Class $(VII_0)$] 
consists of surfaces with $ b_1(S)=q (S)= 1, P_{12} (S)=0$.
\end{enumerate}

In  \cite[Theorem 54]{kodairastructure4} Kodaira shows that, if an  elliptic surface is in the class  $(VII_0)$, 
then $S$ is a Hopf surface (case (\ref{case:Hopf}) above).

It was proven by Miyaoka  \cite{miyaoka} that a properly elliptic surface with even first Betti number
$b_1(S)$ (i.e., in the class $(IV_0)$) is K\"ahler, and  Siu proved the same  for all K3 surfaces \cite{siu}.
 
In contrast,  surfaces with odd first Betti number (class $(VI_0)$) are obviously not K\"ahler.

 The group of numerically trivial automorphisms is defined as
\[
\Aut_\QQ(X) = \ker(\Aut(X)
\curvearrowright
H^*(X, \QQ)).
\]
Obviously this contains the connected component of the identity, $\Aut^0(X)$,
 a finite dimensional complex Lie group according to \cite{bm1}, \cite{bm2}.

In fact, as shown by  \cite{Lie78} and \cite{Fuj78}, the quotient  $\Aut_\QQ(X) / \Aut^0(X)$ is a finite group
whenever $X$ is a compact K\"ahler manifold.

For algebraic elliptic surfaces, we studied this quotient group extensively in \cite{CFGLS24}, \cite{CLS24}.

Turning to the non-algebraic case,
 for  K3 surfaces and tori, the group $\Aut_{\QQ}(S)$ resp.\ the quotient $\Aut_{\QQ}(S)/\Aut^0(S)$ 
 is trivial (\cite{PS72}, \cite{br}), hence in the K\"ahler non-algebraic case  (Class $(IV_0)$)
 we are left with dealing   with the case of properly elliptic surfaces with $b_1(S)$ even
(whence we can build largely on \cite{CFGLS24} and \cite{CLS24}).

If instead $b_1(S)$ is odd and $P_{12} (S) >0$, then  $S$ is either a Kodaira surface 
(and has algebraic dimension equal to $1$) or an elliptic quasi-bundle (see \cite{bpv}, 
section 5  ff.\ of Chapter V for a summary of their properties).

Kodaira surfaces are classified into two types.
A primary Kodaira surface has trivial $K_S$, $b_1(S)=3$, and is an elliptic fibration over an elliptic curve.
A secondary Kodaira surface has a primary Kodaira surface as a finite cyclic unramified covering, has 
$b_1(S)=1$ and $p_g=0$.

We do not attempt here  the full study of the automorphism groups of all the   non-K\"ahler elliptic surfaces.
Instead we only treat the Kodaira surfaces and the  Hopf surfaces,
leaving the elliptic quasi-bundles with odd first Betti number $b_1(S)$  for a future investigation.

%

An important feature is that K\"ahler elliptic surfaces admit algebraic surfaces as deformations, and our 
main result is that the group  $\Aut_{\QQ}(S)$ is constant along deformation families with a
connected base  and with fixed Jacobian surface.

\begin{thm}\label{main}
Let $ f \colon \sS \rightarrow \sB $ be a smooth family with connected base $\sB$ of relatively minimal elliptic K\"ahler surfaces $S_t, t \in \sB$ having the same Jacobian surface $J(S)$.  

If $p_g(S_t)>0, \chi(S_t)>0$, then 
we have equality of the two   groups $ \Aut_{B, \QQ}(S_t) $ and $ \Aut_{ \QQ}(S_t)$, and 
 $  \Aut_{B, \QQ}(S_t) = \Aut_{ \QQ}(S_t)$ is independent of $t \in \sB$.
\end{thm}

It is to be noticed that if $f \colon S \ra B$ is a minimal elliptic K\"ahler surface which is not algebraic,
then necessarily $p_g(S) \geq 1$, and
moreover   $b : = genus (B)  \geq 1$ in  the case where $\chi(S)=0$.

In the algebraic case,
Enriques surfaces (which have $p_g(S)=0$)
can admit numerically trivial automorphisms (cf.\  \cite{BP83}, \cite{MN84}, \cite{Muk10}) which act nontrivially
on the base,
but there are only very few Enriques surfaces admitting such numerically trivial automorphisms
(and similarly for rational elliptic surfaces).

This leaves  open the question whether 
the independence assertion of 
 Theorem \ref{main} may be valid for $\Aut_{B,\QQ}$ or $\Aut_\QQ$ without the assumption $p_g(S) > 0$.

\medskip

\medskip

As in Part II, \cite{CLS24}, a main tool is the Mordell Weil group of the Jacobian surface $J(S)$
and the automorphisms  which it defines on $S$, 
in particular those which induce the identity on the base $B$ and thus act on a general fibre by a translation.

If you look however in the literature, you can see that Kodaira defined the Jacobian surface $J(S)$
only for elliptic surfaces without multiple fibres. Much later, Friedman and Morgan \cite{FM94}
showed the existence, for any elliptic surface, of a {\bf basic elliptic surface}, which equals
the Jacobian surface $J(S)$ if $S$ has no multiple fibres. We find only natural 
to call it in general the Jacobian surface and to denote it $J(S)$.

In the second section we recall some basic results of Kodaira and of  \cite{FM94},
sketching some arguments which are only highlighted there.

Then in the third section we extend several results of  \cite{CLS24}, especially we give explicit
upper bounds for the size of the group $\Aut_{\QQ}(S)$ 
 (detailed in Theorem \ref{thm:CLS} )
and describe its general structure as follows:

\begin{thm}
\label{thm:CLS'}
Assume that $ f \colon S \ra B$ is a K\"ahler elliptic surface with $p_g(S)>0$
(in particular this assumption  holds if $S$ is not  algebraic);
and assume $\chi(S) > 0$.  Then the following hold:

(1) $\Aut_\QQ(S)$ is isomorphic to a subgroup of $\MW(J(S))_\tor$, and, as such, it is a finite 2-generated abelian group
which can be written in the form $ G = \ZZ/d\ZZ \oplus \ZZ/ da\ZZ$.

(2) $\Aut_\QQ(S)$ is trivial  if there is a fibre of additive type
or if there is no multiple fibre with reducible support.
\end{thm}

In the  4th section we describe completely the group $\Aut(S)$ of
automorphisms and  the group $\Aut_{\QQ}(S)$ of numerically trivial automorphisms 
for Hopf surfaces and Kodaira surfaces, both in the primary and in the secondary case
(Theorems \ref{thm:Hopf}, \ref{thm:1st}, \ref{thm:2nd}).

Section five discusses some general facts for elliptic bundles and quasi-bundles, 
borrowing from  \cite[Ch.\ V]{bpv} and  \cite{CFGLS24}.

The  penultimate section describes  completely the group $\Aut_{\QQ}(S)$ for elliptic bundles and quasi-bundles which are K\"ahler (equivalently, those with even first Betti number $b_1(S)$).
The following theorem  collects Theorem \ref{bundle}, Theorem \ref{principal}, the exclusion result of subsection  \ref{ss:not_elliptic}, and Theorem \ref{thm:quasi'}:

\begin{thm}
\label{thm:quasi}
Let $S \ra B$ be a K\"ahler non-algebraic 
 elliptic bundle or quasi-bundle with fibre $E$, 
which is a quotient 
 $S = S'/G$ of a principal elliptic bundle $ S' \colon \ra B'$, chosen minimally.

  If $S \ra B$ is  an elliptic bundle, which is not a torus,  then   $b \geq 2$,    it is a principal bundle and
$\Aut_{\QQ}(S) \cong \Aut^0(S) \cong E$.
  
 If $S$ is not  an elliptic   bundle, then,
 letting $ b= \mathrm{genus}(B)$,
\begin{enumerate}
  \item 
   $G$ acts by translations on $E$ and
  $\Aut^0(S) \cong E$. 
\item
If $ b \geq 2$, then $\Aut^0(S) \cong E$ is a subgroup of  index $\leq 6$  in the group
  $\Aut_{\QQ}(S)$ and $\Aut_{B, \QQ}(S)  = \Aut_{\QQ}(S)$.
\item
If   $b=1$,
 then $G$ is abelian, 
 and
 $ \Aut_{\ZZ}(S) / \Aut^0(S) $ has order at most $2$.
\item
For $b=1$, 
 $ \Aut_{\QQ}(S) / \Aut^0(S) \cong \sT,$ where  $\sT$ is the maximal group of  translations on $B$
which leave invariant  the monodromy of the covering $B' \ra B = B'/G$  (hence the ones which lift to automorphisms of $B'$).
\item
 In particular, for the class of surfaces with $b=1$, which have $p_g(S)=1, \chi (S)=0, K^2_S=0, $ there is no 
absolute upper bound for $|\Aut_{\QQ}(S) / \Aut^0(S) |$, but there is the  relative upper bound $|\Aut_{\QQ}(S) / \Aut^0(S) | \leq P_2(S).$ 
\end{enumerate}

\end{thm}

This completes the classification of $\Aut_\QQ(S)$ for all K\"ahler elliptic surfaces.

\section{The Jacobian of a complex  elliptic surface.}

Let $f\colon S\rightarrow B$ be an elliptic surface, that is, $S$ is a compact smooth complex surface,
$B$ is a compact smooth complex curve ({\em base} curve),
and $f$ is a surjective holomorphic map with connected fibres of arithmetic genus $1$.

We shall throughout assume that the fibration is relatively minimal, that is, there is no $(-1)$-curve
(a smooth rational curve with self-intersection $(-1)$)  contained in a fibre.

As a general notation, we denote by $S^*$ the open set of $S$ where the derivative of $f$ does not vanish,
that is, the complement of the critical set $Crit(f)$; and by $S^0$ the complement of the set of singular fibres,
i.e., $ S^0 : = f^{-1} (B^*)$, where      the complement of $B^*$ is the set of critical values,
also called the
{\bf branch set } of $f$, $   \sB (f) : = f ( Crit(f) )$.

\subsection{The Jacobian surface in the algebraic case}
As recalled in \cite{CLS24}, if  $S$ is algebraic, one  looks at the function field extension $\CC(S) \supset \CC(B)$
as providing a curve $\sC$ of genus $1$ (the fibre of $f$ over the generic point of $B$)
defined over the non-algebraically closed field $ \sK : = \CC(B)$.

One lets $\Aut_{\sK}(\sC)$ be the subgroup of  automorphisms of $\sC$ which act as the identity
on the base field $ \sK$. By base extension,
these embed into $\Aut_{ \bar{\sK}}(\sE)$, where $\sE : = \sC \otimes_{\sK} \bar{\sK}$
is  an elliptic  curve, $\sE$ acts by translations on itself via the isomorphism 
 $ \sE \cong \Pic^1(\sE) \cong \Pic^0(\sE)$ provided by the choice of a point $0 \in \sE$
 (here $\Pic^d(\sE)$ is, as usual,  the
 set of linear equivalence classes of divisors of degree $d$).
Since the base field is of characteristic zero, we infer 
 \begin{equation}\label{autell}
  \Aut_{\bar{\sK}}( \sE) = \sE \rtimes \mu_r, \quad r \in \{2,4,6\},
  \end{equation}
  where $\mu_r$ is the group of $r$th roots of unity in $\CC$,
  and $r=2$ except for the special cases $r=4$ for the Gaussian = harmonic elliptic curve,
  and $r=6$ for the Fermat = equianharmonic elliptic curve.

Going back to $\sC$, there is no longer 
 (as we were taught by Kodaira \cite{kodaira2-3} and Shafarevich \cite[Chapter VII]{steklov})
  an isomorphism 
 $\sC \cong \Pic^0(\sC)$ as curves over the field $\sK$ if there is no $\sK$-rational point on $\sC$, 
 but the isomorphism $\sC \cong \Pic^1(\sC)$ holds true
  (by Riemann--Roch and by descent).
   
 Then one defines  $\Jac (\sC) : = \Pic^0(\sC)$, and $\sC$ is a principal homogeneous space over $\Jac (\sC)$
 (\cite{steklov}, Chapter 2 of \cite{dolg-cime}); 
 in particular, the Mordell--Weil group, the group of $\sK$-rational points of $\Jac (\sC)$,
 acts on $\sC$.
 This envisages
 the torsion subgroup of the Mordell--Weil group as a major source for the construction of numerically trivial automorphisms.
 
 The geometrical counterpart of $ \Jac (\sC) $  is the Jacobian surface $ J(S)$, 
 a suitable compactification of  the sheaf of groups $\sR^1 f_* \hol_S / \sR^1 f_* \ZZ$ 
 (whose stalks have a connected component of the identity which is either an elliptic curve, or $\CC^*$, or $\CC$:
 in the latter cases one says that the fibre is of multiplicative, resp.\ additive type).
 
The Jacobian surface $ J(S)$ has the property
  that the singular fibres of $J(S)$ are  the same as the singular fibres $F$ of $S$ whenever $F$ is not multiple.
  What happens in the case of multiple fibres is explained by Kodaira as follows.
  
  \subsection{Kodaira's covering trick}  
  
  If we have  multiple fibres, the trick of Kodaira is to take a Galois covering of the base with abelian Galois group $G$, $B' \ra B = B'/G$, such that the normalization $S'$ of the pull back  $ S \times_B B'$ is smooth
  and without multiple fibres.
  
  In fact, over a multiple fibre $F = m F_0$ lies a non-multiple fibre $F'_0$ which is an abelian \'etale
  covering of $F_0$. The Jacobian fibration is locally described as a quotient of $S'$ via the action of a cyclic group 
  of order $m$ acting as the identity on the fibre $F'_0$  if the fibre is smooth.

 \subsection{The Jacobian surface in the general, also non-algebraic case} 
 Kodaira (\cite{kodaira2-3}, see also the survey paper \cite{gnsaga} for a quick summary) constructed the Jacobian fibration $J(f)\colon J(S)\rightarrow B$ for  any elliptic fibration
 $f\colon S\rightarrow B$ without multiple fibres.
 
  Kodaira also  showed that under this assumption $S$ is a deformation of $J(S)$:
 indeed, all the elliptic fibrations
 $f\colon S\rightarrow B$ without multiple fibres and with given Jacobian surface $J = J(S)$ are principal homogeneous spaces  
 over $J= J(S)$ and classified by the cohomology group $H^1 (B,\hol_B( J^* ))$ ($ J^*$ is the open set
 of $J(S)$ where $J(f)$ has nonzero derivative) and $\sG : = \hol_B( J^* )$ is the sheaf of holomorphic sections of $B$ into $J^*$ 
  (see Theorem \ref{tate-shaf}).
 
 \begin{remark}
 Since $J(f)$ admits a section $\Sigma$, it follows that {\bf the Jacobian surface $J(S)$ is algebraic}, 
 since it admits a divisor $D$ with $D^2 > 0$: it suffices to take $ D = \Sigma + m F$,
 with $m \in \NN$ sufficiently large.
 \end{remark}
 
 Kodaira's proof was based on the consideration of the functional  and of 
 the homological invariant of $f$.
 
 Letting  $B^*$ be the set of critical values of $f$ as above, then we have a function $j : B^* \ra \CC$
 such that $j(t)$ is the $j$-invariant of the fibre $F_t : = f^{-1} (t)$, 
 normalized such that the elliptic curve with an automorphism of order 4 has $j$-invariant $1$. 
 Kodaira proved that
 this function extends to a meromorphic function   
 $$j : B \ra \CC \;\; \text{ (or a holomorphic map } \;\; j : B \ra \PP^1(\CC)),
 $$
 called the {\bf functional invariant}. 
 
 Letting $B' \subset  B^*\subset  B$ be the set  of points where $ j (t) \neq 0,1, \infty$,
 then the locally constant sheaf $\sR^1f_* (\ZZ)_{|B'}$ is associated to a representation 
 $\rho : \pi_1 (B') \ra SL(2, \ZZ)$, and compatible with lifts of $j$ to the upper half plane, which is called the {\bf homological invariant}.
 
 Then Kodaira shows that, for each pair of compatible invariants $j, \rho$, there is a unique elliptic surface 
 $J$ with a section
 having  these very invariants.
 
 \subsection{The case with multiple fibres}
 One can define the holomorphic function $ j : B^* \ra \PP^1(\CC)$ for any elliptic surface $S$, 
 and one has a homological invariant 
 $\rho : \pi_1 (B') \ra SL(2, \ZZ)$ where $B' \subset B^*$ is  the set  of points where $ j (t) \neq 0,1$.
 
 Friedman and Morgan {\cite[Theorem 3.14, p.\ 45]{FM94}} extend Kodaira's result as follows:
 
 \begin{thm}
Given an elliptic surface $f \colon S\ra B$,
 $j$ has a holomorphic extension to $j \colon B \ra \PP^1(\CC)$, 
 and $f$ determines a homological invariant $\rho$.
For this choice of a compatible pair 
 of holomorphic and homological invariants $(j,  \rho)$,  there exists a unique 
 elliptic surface $J(f)\colon J(S)\ra B$  
  with a section
 admitting this  pair of invariants.
 \end{thm}
 
 A proof is based on Kodaira's trick explained above: after the Galois base change $B' \ra B = B'/G$,
 $S'$ has no multiple fibres and one can take the Jacobian $ J(S')$. Then one can locally glue
 some other $G$-quotient, where $G$ acts as the identity on the fibres lying over the smooth
  multiple fibres of $S$,  whereas for the other multiple fibres the picture is more complicated (cf.\ \cite[p.\ 100]{FM94}).
 
 In Lemma 5.2 of \cite{FM94} the authors give a variant of Kodaira's construction of $J(S)$
 for the case where $S$ has no multiple fibres: there is an open  covering $\{ U_i\}$ of $B$ in the Hausdorff topology,
 and there are local biholomorphisms $\phi_i : S _{|U_i} \cong  J _{|U_i} $ such that,
 on $U_i \cap U_j$, $\phi_i \circ \phi_j^{-1} = : \tau_{ij}$ is given by translation on the fibres; conversely,
 any choice of such $U_i, \phi_i$,  gives  a surface $J \cong J(S)$, well defined up to a translation.

\begin{remark}
Friedman and Morgan call this surface $J(S)$ the basic elliptic surface,
and  create  a different (nonstandard) set of notations, to which we shall not adhere, for the sake of simplicity.
\end{remark}

We take up however their definitions (section 1.5, page 75  and following).

\begin{defin}
Given a Jacobian elliptic surface $J$, that is, an elliptic surface admitting a section, which is taken as the zero section,
then the {\bf Analytic Tate-Shafarevich group}   is defined as 
$$  \Sha ^{an} (J) : =   H^1 (B,\hol_B( J^* ))=  : H^1 (B, \sG) , $$
$\sG : = \hol_B( J^* )$ being  the sheaf of holomorphic sections of $B$ into $J^*$.

It is, by Proposition 5.5 of \cite{FM94},   canonically identified with
 the set of all isomorphism classes of pairs $(S, \psi)$ of an elliptic surface $f \colon S \ra B$
 without multiple fibres, but with some singular fibre, and an isomorphism $\psi : J (S)  \cong J$. 

\end{defin}

The analytic Tate-Shafarevich group contains the algebraic Tate-Shafarevich group, and indeed there is
 an abstract criterion for an elliptic  surface  from $\Sha^{an}(J)$ for being algebraic. 
The second part of the following Theorem  is
due to Kodaira  \cite[Theorem 11.5]{kodaira2-3}, see also pages 86-87, especially Prop. 5.21, of \cite{FM94}.

\begin{thm}\label{tate-shaf}
(1) $  \Sha ^{an} (J) \subset H^2 (J, \hol^*_J)$ is the inverse image of the Kernel $\Lam$ of 
the homomorphism $ H^3(J, \ZZ) \ra H^1(B, \ZZ)$
which is Poincar\'e dual to $H_1(J, \ZZ) \ra H_1(B, \ZZ) $,
under the exact sequence 
$$ H^2 (J, \ZZ) \ra H^2 (J, \hol_J)  \ra H^1 (B, \sG)\ra \Lam \ra 0.$$

Specifically, $\Lam = 0$ if $J$ has some  singular fibre, hence in this case   $H^2 (J, \hol_J)  \twoheadrightarrow H^1 (B, \sG)$ is surjective
and the analytic Tate-Shafarevich group is connected.

(2) Moreover, an element $[S] \in \Sha ^{an} (J) $ is algebraic if and only if it is represented by a point of finite order.
\end{thm}

The most important result for our purposes is the following.

\begin{thm}[{\cite[pp.\ 45-47]{FM94}}]\label{J-acts}
Let $f\colon S\rightarrow B$ be a (relatively minimal) elliptic surface, and let $J = J(S) $
be its Jacobian surface. 

Then the following hold:

\medskip

(i) The holomorphic map $ T_0 (S) \colon  S^0 \times_{B^*}  J^0 \ra S^0$, such that $ (p,\eta) \mapsto div (p + \eta)$,
extends  to a holomorphic map 
 $ S^* \times_B J^* \ra S^*$, and to a meromorphic map over $B$,
 $$ T(S)  \colon S\times_B J(S)\dashrightarrow S.$$ 

 It  has the property that, for every point $t\in B$ corresponding to a smooth fibre $F_t=f^*t$ of $S$, the induced map from 
 $J_t \times F_t \ra F_t$ is just translation given the group structure on the first factor determined by the choice of the zero section. 

(ii) For any section $\sigma \colon B\ra J$, there is an induced automorphism $T_\sigma\colon S\ra S$ over $B$ such that, on a smooth fibre $F_t$ of $f$, $T_\sigma$ acts by translation of $\sigma(t)\in \Pic^0(F_t)$.

\end{thm}
\begin{proof}
For (i), as  discussed briefly  on \cite[pp.\ 46--47]{FM94}, we split the proof into two parts.

The first case is where $S$ has no multiple fibres. Then, by the cited Lemma 5.2 of \cite{FM94}, 
it suffices to show the statement for $S = J$ because $T(J)  $ acts by translations on each fibre,
hence commutes with the isomorphisms $\tau_{ij}$ and induces  $T(S)  $.

The statement for $J$ follows by the algebraicity of $J$, and then we use \cite[Proposition 5.4, p.\ 81]{SS19}.

A simple argument is that $J $ is the minimal resolution of the Weierstrass model $W$:
the group structure map $ J^* \times J^* \ra J^*$ is then restricted to the open set
$ J^{**}$, containing, for the reducible fibres, only the  smooth locus of the
fibre component which intersects the $0$-section of $J$
( see \cite{Mir89}, Lecture III, or \cite{SS19}, 5.7).  $ J^{**}$ has also a group structure, and
the group law extends to a rational map $ W \times_B W \dashrightarrow  W$ by the group law on plane cubics with a given point given as origin.

Since $J$ is a relatively minimal model, $ \mathrm{Bir}(W) = \mathrm{Bir}(J) = \Aut(J)$
 where, as usual, we only consider birational maps and automorphisms respecting the elliptic fibration.

In the case where $S$ has multiple fibres, by Kodaira's trick, we see $ S $ as $S = S' / G$,
and   $J(S) = J(S') /G$. Then one sees that, since $T(S')$ is $G$-invariant, it descends to $T(S)$.

 For (ii), we restrict the meromorphic map $J(S)\times_B S\dashrightarrow S$ from (i) to $S\cong \sigma(B)\times_B S$ to obtain a meromorphic map $t_\sigma\colon S\dashrightarrow S$ over $B$. By the relative minimality of $S$ over $B$, $t_\sigma$ extends to an (everywhere defined) holomorphic map $t_\sigma\colon S\rightarrow S$ over $B$, which is evidently biholomorphic and  is as described in (ii).
\end{proof}

 \begin{remark}
The open set $ J^{**}$ is important also in another direction:  one defines (\cite[p.\ 82]{FM94})
$$ \sJ : = \hol_B ( J^{**} )\subset \sG = \hol_B ( J^* ),$$
the subsheaf of sections which, for reducible fibres, pass through the component containing the intersection with the $0$-section.

Then we have an exact sequence 
$$ 0 \ra \sJ  \ra \sG  \ra \sG / \sJ \ra 0 $$
where $ \sG / \sJ $ is a skyscraper sheaf.
\end{remark}

\subsection{Isomorphism classes of elliptic surfaces with  multiple fibres}

We follow again  \cite{FM94}, section 1.6, especially pages 102-103 and 113. 

Kodaira defined the so called {\bf logarithmic transformation}, which is a transformation replacing a smooth fibre $F_{t_0} $
(or a multiplicative fibre, i.e., an $\I_k$-fibre)  by a multiple fibre.

Kodaira takes a disc neighbourhood $U \cong \De$ of the point $t_0$ and takes the pull back $f'$ of $f$ by a ramified covering 
$\varphi : \De \cong U'  \ra \De \cong U$, $\varphi (z) = z^m$, so that $f^{-1} (U) =  (f')^{-1} (U')/G$,
where $G $ is the cyclic group $\mu_m$. 

Then, given a torsion line bundle $\xi$ on $F_{t_0} $, of order $m$, the $G$-action is twisted by the `translation' 
$\tau_{\xi}$ corresponding to $\xi$ on each fibre (that is, $ p \mapsto p + \xi$).

Then the open set $f^{-1} (U) $ in $S$ is replaced by the quotient of $  (f')^{-1} (U')$ by the new action
of $G$, which is now free, and produces a multiple fibre $F'_{t_0} = m F'$ which is the quotient of
$F : = F_{t_0} $ by the cyclic group generated by $\tau_{\xi}$.  The unramified covering 
$$F = \CC/ \Lambda  \ra F' = \CC / \Lambda' $$
is classified again by a torsion subgroup of $\Pic^0 ( F')$,
cyclic of order $m$,  given by $\frac{1}{m} \Lambda /   \Lambda'$.

Taking instead a fibre of type $I_k$, with $k \geq 2$, we get upstairs $k$ $A_{m-1}$ singularities, and,
resolving them, we get an $I_{km}$ fibre $F$, on which we can take an automorphism of order $m$
acting freely, and which extends to the nearby fibres as a translation of order $m$.

Then dividing by the twisted action we get a multiple fibre $F'_{t_0} = m F'$, where again $F \ra F'$
is an unramified cyclic covering of order $m$, thus determined by a torsion subgroup of $\Pic^0 ( F' )$.

Observe here that we have an exact sequence
\begin{equation}\label{Pic} 1 \ra \CC^* \ra ( \CC^*)^k \ra ( \CC^*)^k \ra \Pic^0 ( F' ) \ra  1,
\end{equation}
whence $\Pic^0 ( F' )  \cong \CC^*$.

\begin{remark}
As the logarithmic transformation creates a multiple fibre, one can also consider the inverse transformation, which
replaces a multiple fibre by a non-multiple one.
\end{remark}

\begin{defin}\label{T}
Given an elliptic surface $ f \colon S \ra B$, and points $t_1, \dots, t_n \in B$
such the fibres $F_{t_j}$ are of type $I_{k_j}$, for $ j = 1, \dots, n$,  let $\xi_j \in \Pic^0( F_{t_j})$
an element of torsion order exactly $m_j$.

Then $T (S, \{t_j\}, \{\xi_j\})$ is defined as  the set of isomorphism classes of elliptic surfaces $f' \colon S' \ra B$ such that 

(1) $J(S) \cong J(S')$
are biholomorphic
via a biholomorphism $\Psi$ such that   $J(f)=  J(f') \circ \Psi$,

(2)  the restrictions of $S'$ and $S$ to $ B \setminus \{t_1, \dots, t_n\}$ are locally (over $B$) biholomorphic
via  biholomorphisms commuting with $f, f'$, 

(3) over some neighbourhood of each point $t_i$ the  surface $S'$  is  obtained 
from $S$  via the generalized logarithmic transformations
associated to the torsion classes $\xi_j$.
\end{defin}

We summarize here Theorems 6.6 and 6.7 of \cite{FM94}:

\begin{thm}
\label{T-family}
(I) Let $ f \colon S \ra B$ be an elliptic surface with multiple fibres $F_{t_j}$ of multiplicity $m_j$
over the points $t_1, \dots, t_n \in B$, and let $J : = J(S)$ be the Jacobian surface of $S$.

Then $S \in T (J, \{t_j\}, \{\xi_j\})$, where the classes $\xi_j$ are uniquely determined.

(II) $ T (S, \{t_j\}, \{\xi_j\})$ is a principal homogeneous space over 
$$  \Sha ^{an} (J) =   H^1 (B,\hol_B( J^* ))=   H^1 (B, \sG) . $$

(III) For each $S' \in T (S, \{t_j\}, \{\xi_j\})$,  $\sR^1f_* \hol_S$ and $\sR^1f'_* \hol_{S'}$
are canonically isomorphic.


\end{thm}

\begin{remark}
Concerning algebraicity of an elliptic surface $S$ with multiple fibres,
%
one can apply  Kodaira's  trick and
obtain a surface $S'$ without multiple fibres, such that $S'/G \cong S$.

Since $S$ is algebraic if and only if $S'$ is algebraic, one can apply then Theorem \ref{tate-shaf}.

\end{remark}

\subsection{Elliptic bundles}
These are the elliptic surfaces $f \colon S \ra B$ with $e(S)=0$ and without multiple fibres
because, by the Zeuthen-Segre formula, 
$$ e(S) = 
\sum_{t \in B} e (F_t) \geq 0,$$ 
equality holding if and only if all fibres have smooth support.

Then the period map $ j : B \ra \CC$
 is constant and all the fibres  are isomorphic, and, by the theorems of Grauert and Kuranishi of 
deformation theory \cite{Grauert74}, \cite{Kuranishi62}, \cite{Kuranishi64}, we have a holomorphic bundle.

A detailed treatment of  these elliptic bundles is given in \cite{bpv}, section 5 of Chapter V, pages 143--148,
and we summarize now the results.

If $E$ is an elliptic curve, we have an exact sequence 
$$ 0 \ra E \ra \Aut(E) \ra \mu_r \ra 0, \;\;\;  r \in \{2,4,6\},$$
and a corresponding exact sequence  of sheaves of local holomorphic maps
$$  0 \ra \sE_B  \ra \sA(E)_B \ra \mu_r \ra 0, \;\;\; r \in \{2,4,6\},$$
and an exact  cohomology sequence 
$$ H^1(\sE_B)  \ra H^1(\sA(E)_B) \ra H^1(B, \mu_r)$$
where $H^1(\sA(E)_B)$ classifies the elliptic bundles, and $ H^1(\sE_B) $ classifies the {\bf principal bundles}.

Representing $ E = \CC/ \Lam$, there is a further cohomology exact sequence
\begin{equation}\label{pb} H^1(B, \Lam) \ra H^1(B, \hol_B) \ra H^1(\sE_B) \stackrel{c}\ra  H^2(B, \Lam) \ra 0,
\end{equation}
 where $\ker(c)$ is the subgroup of principal bundles whose
cocycle $\xi$ is  locally constant. 
 
 \begin{thm}
 \label{bundle}
 i) For a principal bundle $S$, $  c (\xi) =0 \Leftrightarrow $ the bundle $S$ is topologically trivial, hence $b_1(S) = b_1(B) + 2$ and $b_2(S) = 2 b_1(B) + 2$.

  ii) For a principal bundle $S$  with  $  c (\xi) \neq 0 $,  $b_1(S) = b_1(B) + 1$ and $b_2(S) = 2 b_1(B)$.

 iii) If $ b = genus(B)=0$, then an elliptic bundle $S$ over $B$  is either a product or a Hopf surface.
 
 iv)  If $ b = genus(B)=1$, then an elliptic bundle $S$ over $B$  falls into the following cases:
 \begin{enumerate}
 \item
 If the bundle $S$ is principal with $c(\xi)=0$, then $S$ is a torus.
 \item
 If the bundle $S$ is principal with $c(\xi)\neq 0$, then $S$ is a primary Kodaira surface.
 \item
 If the bundle $S$ is not principal, then $S$ is a hyperelliptic surface, hence $S$ is projective
 ($S$ is a free quotient $(E_1 \times E)/G$ of a product of elliptic curves).
 \end{enumerate}
 \end{thm}
 
 \begin{remark}\label{construction}
 For later use, we record the following: writing $B = \HH / \Pi$, with $\Pi = \pi_1(B)$, 
 we can construct an elliptic bundle with fibre $E = \Lambda \otimes \RR / \Lambda$
 over $B$ from a cocycle $\xi \in H^1 (\Pi, \Lambda \otimes \RR / \Lambda)$,
 simply taking the quotient of $ \HH \times  (\Lambda \otimes \RR)$
 for the action of $\Pi$ such that $$\ga (t,z) : = (\ga t, z + \xi(\ga)).$$
 
 The construction generalizes for a more general Fuchsian group $\Pi$
 and yields a quasi-bundle,  as we shall use in Sections \ref{s:e+qe} and \ref{s:qb}.
 \end{remark}

\section{Non-algebraic  K\"ahler elliptic surfaces}

Let  $ f \colon S \ra B$ be an elliptic surface, relatively minimal.

\smallskip

Then  $K_S^2=0$,
by Kodaira's  canonical bundle formula,

\begin{equation}\label{eq: canonical bundle formula}
K_S = f^*(K_B+L) + \sum_{1\leq i\leq s} (m_i-1) F'_i
\end{equation}
where the divisors $m_iF'_i$, $i=1, \dots, s,$  are   the multiple fibres of $f$, $F'_i$ is an  indivisible effective divisor, and $\deg L =\chi(S)$.

Hence, by Noether's formula,
$$e (S) = 12 \chi (\hol_S) = 12 (1-q (S)+ p_g(S)).$$

By the Zeuthen-Segre formula,
$$ e(S) = \sum_{t\in B} (e(F_t)) \geq 0,$$
equality holding if and only if all fibres are smooth or multiple of a smooth curve.

If $e(S)=0$, one says that $f \colon S \ra B$ is an {\bf elliptic quasi-bundle}.

In particular, it follows that 
$$
K_S^2=0, \;\; e(S) \geq 0, \;\; \chi(S) 
\geq 0.$$
\begin{proposition} Let $ f \colon S \ra B$ be a relatively minimal elliptic fibration.

If $S$ is not algebraic, but K\"ahler, then $p_g : = p_g(S) > 0$.  
\end{proposition}
\begin{proof}
To see this, note that  $p_g=0$ implies
  $\NS(S) \cong H^2(S,\ZZ)$; then, by the K\"ahler assumption,
  the class of a fibre $F$ is nontrivial and, by the unimodularity of the
intersection product, 
 there is a divisor $D$ which is  not orthogonal to a fibre $F$, 
 therefore $ (D + mF)^2 > 0$ for some $m \in \ZZ$, hence $S$ is algebraic.
 \end{proof}

Similarly,

\begin{proposition}\label{pg=0nonkaehler} Let $ f \colon S \ra B$ be a relatively minimal elliptic fibration.

If $S$ is not  K\"ahler, and  $p_g(S) = 0$, then  $f$ is an elliptic quasi-bundle,
$B$ has genus $b=0$, and  $b_1(S)=1, b_2(S)=0, q(S)=1$.
\end{proposition}

\begin{proof}
By the same token, it follows that, if $S$ is not algebraic, and $p_g=0$,
then the cohomology class of $F$ must be trivial.

Also, every divisor $D$ must be  orthogonal 
to $F$. If $D$ is effective, then $D$ is contained in a union of fibres.
Writing $D =D_+ - D_-$, with $D_+ ,  D_-$ effective, we see that $\NS(S)$ is generated by the fibre components,
hence its rank equals
$$ \rank (\NS(S)) = \sum_{t\in B} (b_2(F_t)-1).$$
If there are singular fibres, which are reducible, then it follows that all the self intersection
form is strictly negative definite, 
and by a theorem of Donaldson \cite{donaldson}, 
it is diagonalizable.

This contradicts that the self intersection of a vertical divisor $D$ is even 
(by the adjunction formula and Kodaira's canonical bundle formula).

The first conclusion is then that the second Betti number $b_2(S)=0$, and all fibres 
are irreducible. 

Then $e(S) = 2 - 2 b_1(S)$: since $b^+=0$, it follows that $b_1(S)$ is odd.

Since $$ e(S) 
\geq 0 \; \Longrightarrow \; b_1(S)=1 \; \Longrightarrow \; e(S)=0,$$
it follows that  $S$ is an elliptic quasi-bundle, and we have $q=1$ and $b=0$, since
$q = b  + h^0 ( \sR^1f_* \hol_S) = b+1$, because $\sR^1f_* \hol_S$ is trivial.
\end{proof}

\begin{remark}\label{b'geq2}
Valentino Tosatti has been asking the first author whether there are effective
examples of this class of elliptic surfaces which are properly elliptic, hence have $P_{12} \geq 1$
(since by Kodaira's theorem, if $P_{12} =0$, we have an elliptic Hopf surface).

By Kodaira's canonical divisor formula, there are at least  three multiple fibres, and  then the
orbifold fundamental group covering, as we shall see, will yield an elliptic bundle
over a curve of genus at least $2$: because there are no elliptic bundles of Kodaira dimension $1$
over a curve of genus $\leq 1$, see Theorem  \ref{bundle}. 

Since $S$ is not K\"ahler, also any of its unramified coverings is not K\"ahler, hence 
$ S = S' / G$, where $G$ acts freely, and $S'$ is a principal bundle with $c(\xi)\neq 0$
over a curve of genus at least $2$. 
\end{remark}

\bigskip

We continue the discussion subdividing as usual into  two cases according to $\chi (S) = \chi (\hol_S)$
being zero or strictly positive.  

\subsection{Case 0 :}$\chi (S) = \chi (\hol_S) =0.$

We have seen that, if  $e(S)=0$, by the Theorem of Zeuthen-Segre, 
all fibres are smooth or multiple of a smooth elliptic curve.
By  Kodaira's trick, $S = S' /G$, where all fibres are smooth, hence the period map is constant
and all fibres are isomorphic. 

Hence the Jacobian $J'$ of $S'$ is a product $B' \times E$, and we have the analytic Tate-Shafarevich group
$H^1 (B', \sG') $ where $\sG'$ is the sheaf of functions on $B'$ with values in the elliptic curve
$ E =  \CC / \Lambda$.

Assume now $p_g (S) >0$: by Kodaira's  canonical bundle formula \eqref{eq: canonical bundle formula},

\begin{equation*}
K_S = f^*(K_B+L) + \sum_{1\leq i\leq s} (m_i-1) F'_i
\end{equation*}
where the $m_iF'_i$ are as above  the multiple fibres of $f$,  and $\deg L =\chi(S)$,
since here $\deg L = 0$, we infer that $ b : = genus(B)  \geq 1$ and moreover $L$ is trivial if $b=1$,
hence also $\sR^1f_* \hol_S$ is trivial. 

Hence, if  $\chi (S)  =0$  and $p_g >0$, then $ q = 1 + p_g \geq 2$ and 
since $q = genus(B) + h^0 ( \sR^1f_* \hol_S)$, we infer that 
$ b  = genus(B)  \geq 2$ or $b=1$ and 
 $ \sR^1f_* \hol_S $ is trivial.
 
 If there are no multiple fibres and $b=1$, then we conclude that $S$ is an elliptic bundle:
 by Theorem \ref{bundle}, it is either a torus, or a Kodaira surface, which is not K\"ahler,
 or a hyperelliptic surface, which is algebraic.

If we assume that $S$ is K\"ahler and non-algebraic, then $S$ is a complex torus. In particular, $S$ 
has no other numerically trivial automorphisms   than  $\Aut^0(S)$.
 
 If there are no multiple fibres and $b \geq 2$, then $S$ is an elliptic bundle over
 $B$, with even $b_1(S)$.

If there are multiple fibres, then we can replace Kodaira's general trick by a more special 
and canonical construction.

Namely, consider  the orbifold fundamental group exact sequence for $f$, see for instance \cite{barlotti}:
\begin{equation}\label{fund-gr-ex-seq} \pi_1(F) \ra \pi_1(S) \ra \pi_1^{orb}(f) \ra 1,
\end{equation}
where, $m_1, \dots, m_r$ being the multiplicities of the multiple fibres, and $ \ga_1, \dots, \ga_r$ 
being simple geometric loops around the branch points,
\begin{multline*}
 \pi_1^{orb}(f) : = 
 \langle \al_1, \beta_1, \dots, \al_b, \beta_b , \ga_1, \dots, \ga_r \mid \\
\prod_i [\al_i, \beta_i]  \ga_1 \dots \ga_r =1,   \ga_1^{m_1}=  \dots =  \ga_r^{m_r} = 1\rangle .
\end{multline*}

If  $ b\geq 1$ and there are multiple fibres, 
$\pi_1^{orb} (f) $ is a Fuchsian group of hyperbolic type, and we can find an epimorphism
$\pi_1^{orb} (f) \twoheadrightarrow  G $ such that we get an unramified covering $S'$ of $S = S' / G$,
such that $f' \colon S' \ra B' $ is elliptic, with all fibres smooth, and the genus $b' : = genus(B') \geq 2$.

 Then again $S'$ is an elliptic bundle over $B' $ with fibre $ E$ having even first Betti number
$b_1(S')$.

We shall treat  the elliptic quasi-bundles, using the methods of \cite{CFGLS24}, 
in more detail in the final section,
achieving complete results in the K\"ahler case.

%
%

\subsection{Case > 0 :} $\chi (S) = \chi (\hol_S) >0.$

 By Proposition 5.7 of \cite{CLS24},
$\Aut_{\QQ}(S)$ induces a trivial action on the base curve $B$,
that is,  
$$\Aut_{\QQ}(S) =  \Aut_{B, \QQ}(S) $$ 
if  
 the genus $b$ of $B$ is at least $1$, or if $\chi(S)\geq 3$. The argument for the remaining case
$b=0$ and
 $\chi(S)=2$  does not carry over directly; instead
we shall exhibit in the sequel a unified proof for all  cases (using the effect of Torelli's Theorem for K3 surfaces that
 $\Aut_{\QQ}(J(S))$ is trivial).

 \begin{prop}\label{trivialonB}
 If the elliptic surface $S\to B$ has $\chi(\mathcal O_S)>0$, then $\Aut_\QQ(S) = \Aut_{B,\QQ}(S)$  if  $p_g(S)>0$.
 In particular, any non-algebraic   K\"ahler elliptic surface $S$ satisfies $\Aut_\QQ(S) = \Aut_{B,\QQ}(S)$.
 \end{prop}

 \begin{proof}

 In favour of a complete argument, we include and unify the proof of \cite[Prop.\ 5.7]{CLS24}.
 To this end, assume that $\sigma\in\Aut_\QQ(S)$ acts non-trivially on the base $B$, say by $\sigma|_B$.
Since the (trivial) action of $\sigma^*$ on $H^1(S,\QQ)$ restricts to that of $(\sigma|_B)^*$ on $H^1(B,\QQ)$, it follows that
 $b=g(B)<2$.
 More precisely, if $b=1$, then $\sigma|_B$
 acts by translations on $B$, so $S^\sigma=\emptyset$. 
 But then the topological Lefschetz formula gives the contradiction 
 \begin{eqnarray}
  \label{eq:e=0} 
 0 & = & e(S^\sigma) = \sum_{i=0}^4 \mbox{tr} (-1)^i \sigma^*(H^i(S,\QQ))\\
 &  \stackrel{\sigma\in\Aut_\QQ(S)}{=} &
 \sum_{i=0}^4 (-1)^ib_i(S) = e(S) = 12\chi(\mathcal O_S)>0. \nonumber
 \end{eqnarray}
 We shall now try to adapt this argument to the case $b=0$ as well, with a particular view to the situation where $\chi(\mathcal O_S)=2$,
 i.e.\ $J(S)$ is a K3 surface.
 
 Still assuming that $\sigma|_B\neq \id_B$, we find that $\sigma_B$ acts with two fixed points on $B\cong \PP^1$, say $0$ and $\infty$.
Denoting the fibres of $S$ above these points by $F_0, F_\infty$, the analogue of \eqref{eq:e=0} reads
\begin{eqnarray}
\label{eq:e=24}
0< 12\chi(\mathcal O_S) = e(S) = e(F_0^\sigma) + e(F_\infty^\sigma).
\end{eqnarray}
To put this to work, we introduce the notation $F_v$ and $J(F)_v$ for the fibres of $S$ and $J(S)$ above a point $v\in B$ and define
\[
B_0 : = B \setminus \{v\in B; F_v \text{ has smooth support}\} = B \setminus \{v\in B; J(F)_v \text{ is smooth}\}.
\]
Then, using the Zeuthen--Segre formula again,
\begin{eqnarray}
\label{eq:B_0-2}
e(S) = \sum_v e(F_v) \geq e(F_0) + e(F_\infty) + \# B_0 - 2.
\end{eqnarray}
Together with \eqref{eq:e=24}, this gives the bound
\begin{eqnarray}
\label{eq:e^sigma}
e(F_0^\sigma)-e(F_0)+e(F_\infty^\sigma)-e(F_\infty) \geq \#B_0-2.
\end{eqnarray}
Looking at the LHS of this inequality, we use the standard fact that any numerically trivial automorphism $\sigma\in\Aut_\QQ(S)$
necessarily preserves each component of a reducible fibre.
In fact, at a reducible fibre $F_v$, this gives $e(F_v^\sigma)=e(F_v)$.
In comparison, at an irreducible fibre, we have 
\begin{eqnarray}
\label{eq:irred-e}
0 \leq e(F_v^\sigma)-e(F_v) 
 \leq e(F_v^\sigma) \leq 4,
\end{eqnarray}
with the two right-hand-side bounds attained if and only if $F_v$ has smooth support and $\sigma$ acts as an involution, 
but not as a translation (if $F_v$ is singular, then the singular point is fixed and the number of fixed points is at most $3$).

The key input for the right-hand-side of \eqref{eq:e^sigma} is the property that the size $\# B_0$ is governed by $\chi(\mathcal O_S)$.

Consider first  the {\bf semi-stable case} where all singular fibres have multiplicative support
and thus $e(F_v) = \nu_v$,
 the number of irreducible components of $F_v$.
Since $S$ and $J(S)$ share the same invariants and the same fibre types (up to multiplicity), Lefschetz' (1,1)-theorem implies,
since $q(S)=0$,

\begin{eqnarray}
10\chi(\mathcal O_S)  & = & h^{1,1}(S) = h^{1,1}(J(S)) \label{eq:rho}\\
&  \geq & \rho(J(S) )\geq 
2 + \sum_v (\nu_v-1) = e(S) -(\# B_0-2),\nonumber
\end{eqnarray}
 where the second inequality is usually coined as (key application of the) Shioda--Tate formula.
 Writing $e(S)=12\chi(\mathcal O_S)$, we infer that 
 \begin{eqnarray}
 \label{eq:B_0}
 \# B_0-2 \geq 2\chi(\mathcal O_S).
 \end{eqnarray}
 With  $\chi(\mathcal O_S)>0$, the discussion around \eqref{eq:e^sigma} and \eqref{eq:irred-e} implies,
 still for the semi-stable case, that
 at least one of the two fibres, say $F_0$, is irreducible.

 But then \eqref{eq:e=24} returns
 \[
 e(F_\infty^\sigma) = e(S) -e(F_0^\sigma) \geq e(S)-4
 \]
 whence $F_\infty$ has to be reducible (by \eqref{eq:irred-e}), and, in fact,
 \[
 e(F_\infty^\sigma) = e(F_\infty) = \nu_\infty \geq  12\chi(\mathcal O_S)-4
 \]
 gives a direct contradiction to \eqref{eq:rho} since $p_g(S)>0$ and $b=0$, hence $\chi(\mathcal O_S)>1$:
 \[
 10\chi (S) \geq \nu_\infty+1 \geq 12 \chi(S) -3.
 \]
 In conclusion, $\Aut_\QQ(S)=\Aut_{B,\QQ}(S)$ in the semi-stable case with $p_g(S)>0$.

 Turning to the {\bf non-semi-simple case}, we assume that there are $N>0$ additive fibres,
 of which $N_0$ outside $0, \infty$.
 The main difference at an additive fibre $F_v$ is that $e(F_v) = \nu_v+1$.
 This has the effect that \eqref{eq:rho} becomes
 \[
 10\chi(\mathcal O_S) \geq e(S)  -(\# B_0-2) - N
 \]
 and \eqref{eq:B_0} becomes
 \begin{eqnarray}
 \label{eq:B_0'}
  \# B_0-2 + N \geq 2\chi(\mathcal O_S).
 \end{eqnarray}
\begin{claim}
\label{claim:irred}
If $p_g(S)>0$, then $F_0$ or $F_\infty$ has irreducible support.
\end{claim}

\begin{proof}[Proof of Claim \ref{claim:irred}]
If $N<2\chi(\mathcal O_S)$, then $\# B_0-2>0$ by \eqref{eq:B_0'}, and  the claim follows as before.

If $N\geq 2\chi(\mathcal O_S)\geq 4$, then $\#B_0\geq N_0\geq 2$, and we infer the following improvement of \eqref{eq:B_0-2}:
\begin{eqnarray*}
\label{eq:B_0-2'}
e(S) = \sum_v e(F_v) \geq e(F_0) + e(F_\infty) + \# B_0 - 2 + N_0 \geq e(F_0) + e(F_\infty) + N_0.
\end{eqnarray*}
As a  consequence, the bound of \eqref{eq:e^sigma} improves to
\[
e(F_0^\sigma)-e(F_0)+e(F_\infty^\sigma)-e(F_\infty) \geq N_0 \geq 2, 
\]
and we are in business again.
\end{proof}

By Claim \ref{claim:irred}, we may assume that $F_0$ is irreducible.
Then, as before, $F_\infty$ has to be reducible,
and in fact additive, as the semi-stable case  for $F_\infty$
leads to the same contradiction as before.

Hence
\[
\nu_\infty+1 = e(F_\infty^\sigma) = e(F_\infty) \geq 12\chi(\mathcal O_S)-4
\]
translates, using \eqref{eq:rho}, into
\[
 10\chi \geq \nu_\infty+1 \geq 12\chi-4.
 \]
 whence $\chi=2$ and $\nu_\infty=19$.
 By Kodaira's classification of singular fibres, this predicts that $F_\infty$ has type $\I_{14}^*$,
 corresponding to the extended Dynkin diagram $\tilde D_{18}$.
 Then this fibre is not multiple, hence $\sigma$ and $J(\sigma)$ act as the identity on the set  of fibre components.

 Since the zero section on  $J(S)$ is preserved by $J(\sigma)$,
 we infer that $J(\sigma)^*$ acts trivially on 
 $ U \oplus D_{18} \subset \NS(J(S))$, hence also on $\NS(J(S))$.
  
   Indeed, we may observe that we have equality $\NS(J(S)) \cong U \oplus D_{18}$.
  For otherwise $\NS(J(S))$ would be a finite index overlattice of $U \oplus D_{18}$
as the rank bound $\rho(J(S))\leq h^{1,1}(S)=20$ is attained. But then, since $U \oplus D_{18}$ has determinant $-4$,
$\NS(J(S))$ would be even unimodular of signature $(1,19)$,
contradicting the standard fact that the difference of positive and negative eigenvalues of any Gram matrix
of a unimodular even lattice is divisible by $8$.

 Hence, by the Torelli Theorem for K3 surfaces, applied to $J(S)$,
 either $J(\sigma)=\id_{J(S)}$ or $J(\sigma)^*$  acts nontrivially on the transcendental lattice which equals $A_1(-1)^2$
(cf.\ the tables of \cite{SZ01}, for instance).
 In the former case, the action on the base would be trivial, giving the required contradiction;
 in the latter case,
  the action on $A_1(-1)^2$ 
   induces a trivial action on the discriminant group $(\ZZ/2\ZZ)^2$
  because $J(\sigma)^*$ acts trivially on $\NS(J(S))$,
  and the two actions glue to a global action on $H^2(J(S),\ZZ)$.
  This leaves three possibilities for the (diagonal) action on $A_1(-1)^2$:
  \[
  \mbox{diag}(1,-1), \;\;   \mbox{diag}(-1,1),\;\;   \mbox{diag}(-1,-1).
  \]
  Each possibility is excluded by Hodge theory as follows:
  in the first two cases, the transcendental Hodge structure splits over $\QQ$ into two one-dimensional pieces,
  the eigenspaces under $J(\sigma)^*$; but this is impossible because by  Hodge theory the eigenvalues must be complex conjugate
  of each other. 
  In the third case, 
  
$J(\sigma)^*$ acts as multiplication by $-1$ on $H^0(\Omega^2_{J(S)})$.
 But then, again by the Torelli theorem, there is a unique such automorphism on $J(S)$,
 namely the one given by inversion on the generic fibre (which indeed preserves all fibre components of the $\I_{14}^*$ fibre).
 This, however, acts trivially on the base, leading to the same contradiction as before.

 Alternatively, 
 the functoriality of the isomorphism $H^0 (\Omega^2_S ) \cong H^0 (\Omega^2_{J(S)} )$ of Lemma \ref{lem:functorial} 
  implies  that,  since $J(\sigma)^*$ acts non-trivially on $H^0 (\Omega^2_{J(S)} )$  in the last step, 
  also $\sigma^*$ acts nontrivially on $H^2(S,\QQ)$, and we can get rid of the latter case.
 \end{proof}

 \begin{remark}
 The above arguments can also be adjusted to the case $p_g(S)=0$ to give an alternative proof of the complex part of \cite[Thm.\ 6.4]{DM22}.
 \end{remark}

 Having settled the action of $\Aut_\QQ(S)$ on the base curve $B$,
 we continue with two general results
 confirming statements from the algebraic setting in \cite{DM22}, \cite{DL23}, \cite{CLS24}.

\begin{lem}
\label{lem:functorial}
Any automorphism $\s$ of $f : S \ra B$ induces an automorphism $J(\s)$ of $J(f) : J(S) \ra B$,
 and there is a canonical  
 functorial isomorphism between
   $H^0 (\Omega^2_S ) \cong H^0 (\Omega^2_{J(S)} )$.  That is,   for 
   each $\s$, the action of $\s$ on the left hand side is isomorphic to the action of $J(\s)$
   on the right hand side.

Moreover, the action of $\s$ on $H^1(S, \QQ)$ is the same as  the action of $J(\s)$
  on $H^1(J(S), \QQ)$.

\end{lem}

  \begin{proof}
  
  We have already seen that to $\s \in \Aut(S)$ corresponds $J(\s)$, and the action
  on the base curve $B$ is the same.
  
  Moreover, we have an isomorphism between $\sR^1f_* (\hol_S)$ and $\sR^1J(f)_* (\hol_{J(S)})$,
  by the property that $J(S)$ is a compactification of $\sR^1f_* (\hol_S) / \sR^1f_* (\ZZ_S)$.
  
  Hence the action of $\s$ on $H^1(S, \QQ)$ is the same as  the action of $J(\s)$
  on $H^1(J(S), \QQ)$.
  
  Moreover, by  relative duality, see \cite[p.\ 29]{bo-mu2}  and also \cite{kleiman},
  
    $$\sR^1f_* (\hol_S)^{\vee} \otimes (\Omega^1_B)    \cong  f_* (\om_{S}) = f_* (\Omega^2_S) ,$$
    and we have an identical formula for $J=J(S)$, 
  hence there is a natural isomorphism between $ f_* (\Omega^2_{S})$ and $ J(f)_* (\Omega^2_{J})$.

  Taking global sections, 
   $$H^0 (\Omega^2_S ) \cong H^0 (B, \Omega^1_B \otimes \sR^1f_* (\hol_S)^{\vee}) \cong H^0 (\Omega^2_J )  .$$
and we have a functorial isomorphism between
   $H^0 (\Omega^2_S ) \cong H^0 (\Omega^2_{J(S)} )$ (functorial means that, for 
   each $\s$, the action of $\s$ on the left hand side is isomorphic to the action of $J(\s)$
   on the right hand side).
   
   Alternatively, to see that we have this functorial isomorphism  it suffices to show that there is a functorial isomorphism between
   $f_* (\om_{S|B})$ and $f_* (\om_{J|B})$.

   This is a consequence of Kodaira's  canonical bundle formula and of the local description
   of the Jacobian surface $J(S)$ in the neighbourhood of a multiple fibre: because the differential form
   $ dt \wedge dz$ on the tubular neighbourhood of the fibre deprived of the fibre, which has a regular extension 
   in the case of  the Jacobian $J(S)$, acquires some divisor of zeros with support on the fibre
    in the case of $S$, which is 
   described by the canonical bundle formula (section 12 of \cite{ kodaira2-3}), and however killed by taking the direct image.
  \end{proof}

 \begin{prop}  We have a natural isomorphism $H^*(S, \QQ) \cong H^*(J(S), \QQ) $.
 
  Moreover any automorphism $\s$ of $f : S \ra B$ induces an automorphism $J(\s)$ of $J(f) : J(S) \ra B$,
 and 
  if $\s \in \Aut_{\QQ}(S) $ then $ J(\s) \in \Aut_{\QQ}(J(S))$.
  
  
  \end{prop}
  \begin{proof}

  The first statement is based on    the theorem of Mayer-Vietoris, applied to  the decompositions
   $$S = S^0 \cup (\cup_1^s U_i), \ S^0 = \cup_1^s W_j$$
  with the $U_i$'s, $W_j$'s being inverse images of connected open sets in $B$, and 
  tubular neighbourhoods of some fibre, which for 
  each $U_i$  is  a singular fibre $F_{p_i}$.

  We have    corresponding (i.e., the inverse images of the same open sets) decompositions for the Jacobian surface $J : = J(S)$, 
  $$J = J^0 \cup (\cup_1^s V_i), \ J^0 = \cup_1^s W_j,$$
  where the open sets $W_j$ are the same in view of 
  Definition \ref{T}
  and Theorem \ref{T-family} (essentially by Lemma 5.2 of \cite{FM94}). 
  
  We can also assume that all these open sets are homotopically equivalent
  to the support of some fibre, hence either a smooth elliptic curve, or a fibre with first cohomology 
  of rank $\leq 1$.
  
   Since  $U_i$, respectively $V_i$, and $W_j$, are homotopically equivalent to a fibre, 
  their  cohomology algebras have degree at most $2$ and their second cohomology groups $H^2(U_i, \QQ), \hdots$
    are  generated by the 
  fundamental classes of the components
  of the fibres. 
  
  If there are no  multiple fibres, again  by Lemma 5.2 of \cite{FM94} we can take also the
  $ U_i = V_i$, and observe that the gluing maps for $S$, and those for $J$, just differ by translations 
  on the smooth fibres: hence, by homotopy equivalence,  the Mayer-Vietoris sequence for $S$ is isomorphic to
  the Mayer-Vietoris sequence for $J$,  and we have proven that $H^*(S, \QQ) \cong H^*(J, \QQ) $.

  If instead there is some multiple fibre, then we use the fact that  $U_i, V_i$  are related via a logarithmic transformation.
  The effect of it is to glue together the same reduced curve, and on the punctured open set
  we have again some  translations: whence we have in any case
  a natural  isomorphism $H^*(S, \QQ) \cong H^*(J, \QQ) $.
  
  Let us pass to the second statement: by the previous Lemma \ref{lem:functorial}
 it is sufficient   to see  that the respective actions on $H^2(S, \QQ) $ and $H^2(J, \QQ) $
 are the same: hence the main idea is to compare locally the action on
 corresponding open sets.

We use the assumption that  the automorphism $\s$ is numerically trivial.
 Recall that the only fibres which can be permuted by $\s \in \Aut_{\QQ}(S)$ are those whose reduced support 
  is irreducible: because, 
   if $\s \in \Aut_{\QQ}(S)$, and the number of components is at least $2$,
  then each component has negative self-intersection, hence it is left invariant by $\s$.

If $\s$ permutes some $U_i$'s , then necessarily $J(\s)$  permutes accordingly the $V_i$'s
and the action on the second cohomology will always take the fundamental class of the fibre
to itself. This solves the question if the corresponding fibre is smooth (in this case, one sees directly that the action
on the first cohomology is trivial on both sides, being induced by a translation), or irreducible but not $I_1$.
In the case of $I_1$,  either the automorphism $\s$  is homotopic to the identity,
or exchanges the two branches at the node.

 In the latter case, it induces an inversion in $\CC^*$ sending $ 0 \mapsto \infty$, and which has a fixed point
 on the smooth part of the fibre: letting this point correspond to the point in the Jacobian fibre
 lying on the zero section, we  have the same action on $\Pic^0( \CC^*)$, hence
also the action on the first cohomology is the same.

If $\s$  leaves a singular fibre fixed, then   the components of the fibre in $S$ are left invariant.

 And $J(\s)$  does not permute the components of the fibre in $J$, which come from the natural compactification of
  $\Pic^0(F^*)$.

 We turn now to fibres which have at least  two  irreducible components, and 
  we need to show that, if a fibre of $J(S)$
   has several components,
  then $J(\s)$ does not permute these components.

  Now, the only fibres of $J(S)$ admitting automorphisms possibly permuting the irreducible components nontrivially
  correspond to the fibres of type of additive type $A_2, E_6, D_n \; (n\geq 4)$ of $S$  or to the fibres of type $m I_k$, with $k \geq 2$.
  
 In  the case of these additive fibres  there are three or four simple components  in the fibre of the Jacobian,
 which correspond to $\Pic^0$ of the end (= simple) components of the additive fibre    in $S$. Hence this case is clear.
  
  If  the fibre $F'' = m F' $ is of type $m I_k$, with $ k \geq 2$, then the corresponding fibre in the Jacobian fibration is of type $I_k$,
  and the component of this fibre which intersects the $0$-section 
  is obviously preserved.
  
  Since $\s$ preserves each component of the multiple fibre, and since by equation \eqref{Pic} 
  we see that the action of $\s$ centralizes  the subgroup of order $m$ corresponding to the
  covering $F \ra F'$, we see that there is a lift of $\s$ which acts as  the identity on the components of $F$:
  hence, collapsing the chains of $A_{m-1}$ singularities, we obtain that $J(\s)$ is also acting
  as  the identity on the components of the $I_k$-fibre of $J(S)$.
  In particular, it acts trivially also on the first cohomology group.

  Hence, finally,  the Mayer-Vietoris exact sequence allows us to conclude  that,
  via the isomorphism of cohomology groups established in the first part of the proof, then
  $\s \in \Aut_{\QQ}(S)$ implies  $J(\s) \in \Aut_{\QQ}(J(S))$.
  \end{proof}

\medskip

We shall now establish the analog of Proposition 3.4 of \cite{CLS24}:

\begin{prop}\label{lem: AutQ vs MW}
Let $f\colon S\rightarrow B$ be a relatively minimal elliptic K\"ahler surface with 
$ \chi(S)>0$ and 
$p_g(S)>0$: then we have
\[
 \Aut_{B, \QQ}(S) \subset \mathrm{Im} (\MW(J(S))_\tor).
\]
In particular,  if $f$ admits a fibre of additive type, then $\Aut_{B,\QQ}(S)$ is trivial.
\end{prop}

\begin{proof}

By the K\"ahler assumption,  $ H^0(\Omega^2_S))$ consists of closed holomorphic forms,
and yields a subspace of the De Rham cohomology group $H^2(S, \CC)$. Hence, if $\s \in  \Aut_{B, \QQ}(S) $,
it acts as the identity on this space.

Since $p_g(S)>0$, there is a nonidentically vanishing form $\om \in H^0(\Omega^2_S))$.

Take a point $p \in S$ where $\om (p) \neq 0$, and such that the fibre through $p$ is smooth.

Then we have local coordinates $(t,z)$ at $p$, where $t$ is a local coordinate on $B$,
and $z$ is a uniformizing coordinate for the fibres.

Writing $ \om = \psi (t,z) dt \wedge dz$, since the tangent bundle to the fibre is trivial,
$ \psi (t,z)  =  \Psi (t) $ and since $\s (t,z) = (t, \s_2(t, z)) = (t, \al z + w(t))$,
we see that $$ \s^* (\om ) = \al \ \Psi (t) dt \wedge dz. $$
Since $ \s^* (\om ) = \om$,  and the function $\Psi (t) \neq 0$, then $\al=1$, and $\s$ acts on the general smooth
fibres via a translation.

The same argument shows that $\s$ acts as a translation on the complement of the critical set of $f$.

Hence $\s$ determines a holomorphic section $\tau_*$ of $J(S)$ on the open set $B^*$.

 We use now the hypothesis $\chi(S) >0$ to infer that $S$ admits no holomorphic vector fields,
 therefore, by the theorem of Fujiki \cite{Fuj78}, $\Aut_{B, \QQ}(S) $ is a finite group.
 
 Hence, our section $\tau$ is a torsion section, let us say, of order $k$.
 
 Hence it takes values in the curve $J(S)[k]$, the closure of the torsion subgroups of order $k$ 
 of the fibres of $J(S)$. Then the closure of the graph of $\tau_*$ is not dense, and, 
 by the Theorem of Remmert and Stein, $\tau_*$ extends to a holomorphic section $\tau$ on the whole of $B$.
 
 By construction, translation by $\tau$ induces the automorphism $\sigma$ on $S$,
 whence the stated inclusion $ \Aut_{B, \QQ}(S) \subset \mathrm{Im} (\MW(J(S))_\tor).$

 For the second assertion, we just  
 use that any numerically trivial automorphism preserves  each 
fibre component.
For a fibre of additive type
of $J(S)$,
the group structure on the smooth locus of any simple fibre component is the additive group $\mathbb G_a$.
Since this has trivial torsion subgroup, it
 follows that 
 the torsion order of $\tau$ is $1$, and $\s$ is the identity.
 \end{proof}

We can now state one of our main results (Theorem \ref{main} in the introduction);
it will have immediate applications.

\begin{thm}\label{main'}
Let $ f \colon \sS \rightarrow \sB $ be a smooth family with connected base $\sB$ of relatively minimal elliptic K\"ahler surfaces $S_t, t \in \sB$ having the same Jacobian surface $J(S)$.  

If $p_g(S_t)>0, \chi(S_t)>0$, then 
we have equality of the two   groups $ \Aut_{B, \QQ}(S_t) $ and $ \Aut_{ \QQ}(S_t)$, and 
 $  \Aut_{B, \QQ}(S_t) = \Aut_{ \QQ}(S_t)$ is independent of $t \in \sB$.
\end{thm}

\begin{proof}
We have seen in Proposition \ref{lem: AutQ vs MW} 
that, under our assumptions, $ \Aut_{B, \QQ}(S) \subset \mathrm{Im} (\MW(J(S))_\tor).$

Presently, however, the right hand side is  independent of $t \in \sB$.

  By Ehresmann's Theorem, first of all, since $\sB$ is connected, all the surfaces $S_t$ are diffeomorphic,
  secondly, 
 if a translation $\tau \in \MW(J(S))_\tor$ yields a transformation in $ \Aut_{B, \QQ}(S_t)$
for one $t$, it does so for all $t \in \sB$.

Thus $ \Aut_{B, \QQ}(S_t)$ is independent of $t \in \sB$, and 
the last claim follows from Proposition \ref{trivialonB}.
\end{proof}

\medskip

Using Theorems 1.1, 1.2 of \cite{CLS24}, we deduce from the previous result:

\begin{thm}[Theorem \ref{thm:CLS'}]
\label{thm:CLS}
Assume that $ f \colon S \ra B$ is a K\"ahler elliptic surface with $p_g(S)>0$
(in particular this assumption  holds if $S$ is not  algebraic);
and assume $\chi(S) > 0$.  Then the following hold:

(1) $\Aut_\QQ(S)$ is isomorphic to a subgroup of $\MW(J(S))_\tor$, and, as such, it is a finite 2-generated abelian group
which can be written in the form $ G = \ZZ/d\ZZ \oplus \ZZ/ da\ZZ$.

(2) $\Aut_\QQ(S)$ is trivial  if there is a fibre of additive type
or if there is no multiple fibre with reducible support.

(3) There are upper bounds
\begin{enumerate}
\item[(i)] In terms of the irregularity $q(S)$, there is the global bound 
\begin{eqnarray}
\label{eq:Aut_Q<=}
|\Aut_\Q(S)|\leq 12\pi^2(q(S)+2).
\end{eqnarray}
\item[(ii)]
If $p\mid|\Aut_\Q(S)|$ for some prime $p>5$, then 
$$p_g\geq  \frac{p^2-1}{12}-\frac{p-1}2 \;\;\; \text{
and } \;\;\; \frac{p^2-1}{24}\mid\chi(S).
$$
For $p\in\{2,3,5\}$, there are properly elliptic surfaces $S$  with any given $\chi>1$ or $p_g>0$ such that $p\mid|\Aut_\Q(S)|$.
\item[(iii)]
If $(\ZZ/p\ZZ)^2\subset\Aut_\Q(S)$  for some prime $p\geq 3$,
then 
$$p_g\geq \frac 1{12}(p-3)(p^2-1) \;\;\; \text{ and } \;\;\;
 \frac{p(p^2-1)}{24}\mid\chi(S).
$$
\end{enumerate}

\end{thm}

\begin{proof} 
 
If $S$ is algebraic, then the results can be found in  \cite[Thm.\ 1.1]{CLS24},
with the only extension that loc.\ cit.\ states the second case of (2) only 
if there is no multiple fibre with singular support.
However, one easily checks that the arguments extend to allow for multiple fibres of type $m\I_1$.

Hence we may assume that $S$ is not algebraic.
Then 

(1) and the statement about the additive fibre from (2) follow from Proposition \ref{lem: AutQ vs MW},
so we turn to the remaining claims.

Kodaira proved  the density of the locus of algebraic elliptic surfaces
in the Kuranishi family of deformations of each elliptic surface $S$ with even first Betti number, see Theorem 7 of \cite{kodstruct-1-2}.
However, we need more specific results.

If $S$ has no multiple fibre, then, by Theorem \ref{tate-shaf},
$S$ lives in the principal homogeneous space $\Sha ^{an} (J(S))$
which at the same time contains its Jacobian $J(S)$.
Hence Theorem \ref{main'} allows us to deduce
\[
\Aut_\QQ(S)=\Aut_\QQ(J(S)) = \{1\}
\]
as stated in (2) from \cite[Thm.\ 1.1 (iii)]{CLS24}.
(3) thus also holds trivially in this case.

If $S$ has a multiple fibre, then 
Theorem \ref{T-family} leads us to consider
the surfaces $S$ with fixed Jacobian $J$ and with multiple fibres of fixed positions and orders.

 Precisely, we have that $T (S, \{t_j\}, \{\xi_j\})$ is a principal homogeneous space for 
 $  \Sha ^{an} (J) =   H^1 (B,\hol_B( J^* ))=   H^1 (B, \sG) , $
 which is a connected manifold. This gives a  smooth 
 family with connected base $\sB$ containing our given surface $S$,
 to which we apply Theorem \ref{main'}.
 
 It turns out, as stated in Theorem 6.12 of \cite{FM94}, that the algebraic surfaces 
 are dense in $\sB$, compatibly with  (2) of Theorem \ref{tate-shaf}.

 Hence (2) follows from  \cite[Thm.\ 1.1 (iii)]{CLS24} applied to any algebraic surface $S_t$ for some $t\in\sB$,
 since $\Aut_\QQ(S)=\Aut_\QQ(S_t)$ by Theorem \ref{main'}.
 For the same reason, the bounds from Theorem 1.2 of \cite{CLS24} carry over as recorded in (3). 
\end{proof}

\begin{remark}
It also follows from the above arguments that for any $d$ and for any squarefree $a$,
there are non-algebraic K\"ahler elliptic surfaces $S$ with
$\Aut_\QQ(S) =  \ZZ/d\ZZ \oplus \ZZ/ da\ZZ$.
\end{remark}

\section{The non-K\"ahler case}

If instead $b_1(S)$ is odd and $P_{12} (S) >0$, then  $S$ is either a Kodaira surface 
(and has algebraic dimension equal to $1$) or an elliptic quasi-bundle.

 Recall that a primary Kodaira surface has trivial $K_S$, $b_1(S)=3$, and is an elliptic fibration over an elliptic curve.
A secondary Kodaira surface has a primary Kodaira surface as a finite unramified covering, has 
$b_1(S)=1$ and $p_g=0$.

The only remaining non-K\"ahler elliptic surfaces are
the elliptic Hopf surfaces
(their first Betti number is odd, $b_1(S) =1$)
 to which we turn now.
 
 See \cite{Namba} for a general account 
 and \cite{cattaneo} for similar results, but different proofs.

\subsection{Automorphisms of   Hopf surfaces}

Assume that $S$ is a primary Hopf surface: then $S = \CC^2 \setminus \{0\}/ \ZZ$ is diffeomorphic to
$S^1 \times S^3$, hence the cohomology algebra 
$$ H^*(S, \ZZ) = (\ZZ [x_1]/ (x_1^2) ) \wedge ( \ZZ [x_3]/ ( x_3^2)),$$
where $x_j$ has degree $j$.

Hence  all automorphisms are numerically trivial, and actually act as  the identity on  $H^*(S, \ZZ) $,
so we continue to calculate $\Aut(S)$.

The group of covering transformations is generated by a contraction
$$ \phi (z_1, z_2) = (\al_1 z_1 + \la z_2^m, 
\al_2 z_2), \ {\rm with} \ 0 < |\al_1| \leq  | \al_2|  < 1, \  (\al_1 -  \al_2^m)\la =0.$$

If $g \in \Aut(S)$, then $g$ lifts to a holomorphic automorphism $\psi$ of $\CC^2$
with $\psi(0)= 0$ and functional equation
$$ \psi \circ \phi = \phi \circ \psi \;\;\; \text{ or } \;\;\;  \psi \circ \phi = \phi^{-1} \circ \psi.
$$

 We rule out now the second functional equation: because $\phi$ is a contraction, in fact
$ \lim_{n \ra + \infty} \phi^n (z) =0, \ \forall z.$ 

In contrast, $ \lim_{n \ra + \infty} \phi^{-n} (z) = \infty, \ \forall z.$ 

Now, if $\psi \circ \phi \circ \psi ^{-1}= \phi^{-1},$ then $ \lim_{n \ra + \infty} \phi^{-n} (z) = \psi(0)$,
an obvious contradiction.
Hence it remains to consider the first functional equation.

Let us consider now the case $\la =0$, and look at the first functional equation
$$ (\psi_1 (\al_1 z_1,  \al_2 z_2),  \psi_2 (\al_1 z_1,  \al_2 z_2)) = (\al_1 \psi_1(z_1, z_2) , \al_2 \psi_2 (z_1, z_2)) .$$
Now $ \psi_1 (\al_1 z_1,  \al_2 z_2) = \al_1 \psi_1(z_1, z_2)$; writing $\psi_1$ as a power series
$$ \psi_1(z_1, z_2) = \sum_{n,r} c_{n,r} z_1^n z_2 ^r ,$$
amounts to 
$$\al_1 c_{n,r} = \al_1^n \al_2 ^r c_{n,r}   \Leftrightarrow c_{n,r} \al_1 (1 - \al_1^{n-1} \al_2 ^r ) =0,$$ for all $n,r$.

Since $\alpha_1\neq 0$, the coefficients $c_{n,r} $ are  zero for $ n \geq 1$ 
unless $ n=1, r=0$ otherwise $ |\al_1^{n-1} \al_2 ^m | < 1$;  
for $n=0$, if and only if $ r$ is such that $ \al_2^r = \al_1$.

 The same argument for $\psi_2$  with coefficients $c'_{nr}$ yields
$$ \psi (z) = ( c_{1,0} z_1 + c_{0,r} z_2 ^r ,  c'_{0,1} z_2 )  \ {\rm if} \ \al_1 = \al_2^r, r \geq 2,$$
$$ \psi (z) = ( c_{1,0} z_1 + c_{0,1} z_2  , c'_{1,0} z_1 + c'_{0,1} z_2 ) \ {\rm if} \ \al_1 = \al_2.$$

This shows that 

\begin{prop}\label{hopf0}
For a primary Hopf surface $S$ with $\la=0$, in particular for a primary elliptic 
Hopf surface, $\Aut(S) = \Aut^0(S)$.

 Indeed, 
 \[
  \Aut(S) =
 \begin{cases}
\mathrm{GL}(2, \CC)/ \ZZ & \text{ if } \al_1 = \al_2,\\
 ((\CC^* \times \CC)\rtimes \CC^*)/ \ZZ
 & \text{ if } \al_1 = \al_2^r, r \geq 2,\\
  (\CC^* \times \CC^*) / \ZZ & \text{ otherwise.}
  \end{cases}
\]
\end{prop}

\begin{proof}
$\ZZ$ is the subgroup generated by  $\phi$ in the  group $\CC^* \times \CC^*$, 
$$\{ (\al_1^k, \al_2^k)| k \in \ZZ\}.$$
\end{proof}

In the  case $\la \neq 0$,
 we write $\alpha_2=\alpha, \alpha_1=\alpha^m$.
 
Then the first functional equation yields 
$$  \psi ( \phi (z)) = \psi (\al ^m z_1 + \la z_2^m, \al z_2) = (\al ^m \psi_1 + \la \psi_2^m, \al \psi_2),
$$
hence $$  \psi_2 (\al ^m z_1 + \la z_2^m, \al z_2) = \al \psi_2 (z_1, z_2)
$$
and $\frac{ \partial \psi_2} { \partial z_1}$ satisfies the functional equation 
$$ \frac{ \partial \psi_2} { \partial z_1} (\phi (z)) = \al^{1-m} \frac{ \partial \psi_2} { \partial z_1}(z),$$


From the above two equations it follows that 
$$ \frac{ \partial \psi_2} { \partial z_1}(z)  \psi_2 (z)^{m-1}$$
is $\phi$-periodic, hence it is a constant $c\in\CC$.

Therefore, by integration,
\begin{eqnarray}
\label{eq:psi_2^m}
\psi_2 (z)^m = c z_1 + d (z_2).
\end{eqnarray}
Write now $\psi_2$ as a power series in $z_1$, and define $g$ as follows:
$$\psi_2 = \sum_{j=0}^{\infty} a_j (z_2) z_1^j , \ g : = \sum_{j=1}^{\infty} a_j (z_2) z_1^{j-1} \Rightarrow  \psi_2 = a_0 + z_1 g.$$

Equating the two power series in \eqref{eq:psi_2^m},
$$ c z_1 + d (z_2) = (a_0 + z_1 g) ^m = a_0^m + m a_0^{m-1}z_1 g + z_1^2   {m \choose{2}}a_0^{m-2} a_1^2 + z_1^3 (...)$$

we first  infer that 

$$ d = a_0^m, \ c = m a_0^{m-1} a_1, \ {m \choose{2}}a_0^{m-2} a_1 = 0.$$

The latter equality yields three possibilities

\begin{enumerate}
\item
$m=1$, 
\item
$a_1 \equiv 0,  m\geq 2,$
\item
$a_0 \equiv 0,  m>2$.
\end{enumerate}

In case (1), assuming $m=1$, then $a_0 = d, \ g  = c$, hence $\psi_2 = a_0 + c z_1$. 

We must have  $\al \psi_2 = \psi_2 (\al z_1 + \la z_2 , \al z_2)$, that is,
$$ \al (a_0(z_2) + c z_1) = c (\al z_1 + \la z_2) + a_0 (\al z_2) \Leftrightarrow \al a_0(z_2) = c \la z_2 + a_0 (\al z_2) .$$
Expanding $a_0= \sum_k c_k z_2^k$ as a power series of $z_2$, the last equation implies that $a_0 = c_1 z_2$, and 
$ \al c_1 + c \la-  \al c_1 =0 \Rightarrow c=0$, hence finally $$\psi_2 = c_1 z_2.$$

In case (2), we claim that all $a_j \equiv 0$, for $ j \geq 1$.
This can be seen by induction, taking the first $a_j$ which is not $\equiv 0$.

Because then, modulo terms of order at least $j+1$,
$$ \psi_2^m \sim  a_0^m + m z_1^j a_j \sim c z_1 + d (z_2) \Rightarrow a_j \equiv 0,$$
a contradiction.

Whence, $\psi_2 = a_0$, and $c=0$, and we conclude again that $\psi_2 = c_1 z_2.$

Finally, in case (3), the assumptions $a_0 \equiv 0$ and $ m > 2$ imply $c=0, d \equiv 0$, hence $$\psi_2 \equiv 0.$$
 
Note that $\psi$ being invertible excludes case (3) and implies $c_1\neq 0$ in cases (1), (2).

\smallskip

We return to the functional equation, which now reads   as

$$  \psi_1 ( \phi (z)) = \psi_1 (\al ^m z_1 + \la z_2^m, \al z_2) = \al ^m \psi_1 + \la c_1^m z_2^m,
$$
hence 
 $$ \frac{ \partial \psi_1} { \partial z_1}(z)  $$
is $\phi$-periodic, hence it is a constant.

Integrating, we get $$  \psi_1  = c' z_1 + d_1 (z_2),$$

and then the functional equation yields

$$  d_1( \al z_2) = \al ^m  d_1 (z_2)  + \la ( c_1^m -c') z_2^m,
$$
Writing $d_1 (z_2)$ as a power series of $z_2$, $d_1 (z_2) = \sum_j \be_j z_2^j$,
we obtain that $\be_j=0$, except for $j=m$, where we must have 
$ c' = c_1^m$, and we conclude that 
$$ \psi_2 = c_1 z_2, \  \psi_1  = c_1^m z_1 + \be_m z_2^m \;\;\;  (c_1\neq 0).$$

It follows then:

\begin{prop}\label{hopf1}
For a primary Hopf surface $S$ with $\la \neq 0$,   $\Aut(S) = \Aut^0(S) = (\CC \times \CC^*) / \ZZ$.

\end{prop}
\begin{proof}

Our calculations have shown that $$\psi (z) = (c_1^m z_1 + \be_m z_2^m,  c_1 z_2)
\;\;\;  (c_1\neq 0).$$ 

These transformations  yield a Lie group which is 
an extension    $\CC \rtimes \CC^*$, and an easy calculation shows that it is $\CC \times \CC^*$
with isomorphism given by associating $ (\frac{\be_m}{c_1^m}, c_1)$ to the given transformation.

Whence, every such transformation normalizes $\phi$, which is an element of this group.
\end{proof}

Kodaira proved in Theorem 32 of \cite{kodairastructure2} 
that every secondary Hopf surface $S$ is a quotient of a primary Hopf surface $\Sigma$
by a finite cyclic group $H$ of order $l$ 
generated by a transformation of the form 
$$ (z_1, z_2) \mapsto (\e_1 z_1, \e_2 z_2), \  {\rm with} \ (\e_1 - \e_2^m) \la =0,$$
with $\e_1, \e_2 \in \mu_l = \{ \e \in \CC| \e^l = 1\}$. 

It follows immediately  that 
\begin{thm}
\label{thm:Hopf}
For Hopf surfaces $\Aut(S) = \Aut^0( S) $ and it is an Abelian Lie group of dimension $2$.

For a primary Hopf surface 
$$\Aut(S) = \Aut^0( S) \cong 
\begin{cases}
(\CC \times \CC^*) / (\ZZ ), & {\rm for }\; \la \neq 0,\\
 GL(2, \CC)/ \ZZ  & {\rm for } \;\la = 0, \al_1 = \al_2,\\ 
 ((\CC^* \times \CC)\rtimes \CC^*)/ \ZZ& {\rm for } \;\la = 0,   \al_1 = \al_2^r, r \geq 2,\\
 (\CC^* \times \CC^*) / \ZZ & {\rm for } \;\la = 0,  \al_1 \neq \al_2^r. 
\end{cases}
$$

For a secondary Hopf surface 
$$\Aut(S) = \Aut^0( S) \cong 
\begin{cases}
(\CC \times \CC^*) / (\ZZ \times (\ZZ/l)), & {\rm for }\; \la \neq 0,\\
(\CC^* \times \CC^*) / (\ZZ \times (\ZZ/l)), & {\rm for } \;\la = 0,   \al_1 \neq \al_2^r,\\
 (\CC^* \times \CC^*) / (\ZZ \times (\ZZ/l)), & {\rm for } \;\la = 0,   \e_1 \neq \e_2,\\
 GL(2, \CC) / (\ZZ \times (\ZZ/l)), & {\rm for } \;\la = 0,   \al_1= \al_2, \ \e_1=\e_2\\
 (\ZZ \times (\ZZ/l)), & {\rm for } \;\la = 0, \al_1 = \al_2^r, r \geq 2, \ \e_1=\e_2.\\
\end{cases}
$$
\end{thm}

\begin{proof}

For primary Hopf surfaces, we simply reproduced the statements of Propositions \ref{hopf0} and \ref{hopf1}.

For a secondary Hopf surface $S = \Sigma/H$, the unramified covering $\Sigma \ra S$ is characteristic, hence $\Aut(S)$ is the quotient 
$N_H / H$ of the normalizer $N_H$
in $\Aut(\Sigma)$ of the subgroup $H \cong (\ZZ/l)$. 

If $\Aut(\Sigma)$ is abelian, then $N_H / H \cong \Aut(\Sigma)/H$.

Otherwise, if the eigenvalues $ \e_1 \neq \e_2$, then the normalizer consists of diagonal matrices,
and everything if  $\e_1 = \e_2$.
\end{proof}

\subsection{Automorphisms of Kodaira surfaces}
Kodaira  proved in \cite{kodairastructure1}  the following characterization of the so called
primary Kodaira surfaces:

\begin{thm}\label{primary}
Let $S$ be a non-K\"ahler surface with trivial canonical divisor.

Then its universal covering is $\CC^2$ and $S = \CC^2 /\Ga$,
where $\Ga$ is a group of affine transformations of the form
$$ \ga(z_1, z_2) =   (z_1 + \alpha, z_2 + \bar{\alpha} z_1 + \beta).$$
The discontinuous group $\Ga$ acts freely and has four generators $\ga_1, \ga_2, \ga_3, \ga_4$
such that 
$$ \alpha_1 =  \alpha_2 =  0, \ \bar{\alpha_3} \alpha_4 -  \bar{\alpha_4} \alpha_3 =  m \beta_2 \neq 0,$$
for some $m$  $\in \NN$, 
and $\beta_1, \beta_2$ are $\RR$-linearly independent (and same for $\alpha_3, \alpha_4$ in view of the previous formula).

The projection $(z_1, z_2) \mapsto z_1$ induces an elliptic bundle over the elliptic curve
$B : = \CC / (\ZZ \alpha_3 + \ZZ \alpha_4)$ with fibre the elliptic curve
$E : = \CC / (\ZZ \beta_1 + \ZZ \beta_2)$. This fibration is the Albanese map.
\end{thm}

 Andrea Cattaneo, \cite{cattaneo} has shown that the group $\Aut(S) $
 lifts to a group of affine transformations of $\CC^2$, and has given a precise description of $\Aut(S)$,
 which we shall now explain   in (6) and (7).
 
\begin{enumerate}
\item
By elementary covering spaces theory, the group $\Aut(S) $
is the quotient $N_{\Ga} /\Ga$, where $N_{\Ga} $ is the normalizer 
of $\Ga$ in the group of biholomorphisms of $\CC^2$.
\item
The fibre bundle homotopy exact sequence is
$$ 1 \ra \Lam_E : = \langle \ga_1, \ga_2\rangle \cong \ZZ^2 \ra \pi_1(S) \ra \Lam_B \cong \ZZ^2 \ra 1,$$
and taking the  Abelianization of $\pi_1(S)$ we get the exact sequence
$$ 0 \ra \ZZ \ga_1 \oplus (\ZZ/m)\ga_2 \ra H_1(S, \ZZ) \ra \Lam_B \cong \ZZ^2 \ra 0.$$
\item
$\Aut(S) $ contains $E$, acting by the translations 
$$ (z_1, z_2) \mapsto (z_1, z_2 + w_2) ,$$ which commute with $\Ga$; indeed 
$E \subset \Aut(S) $ is equal to  $\Aut^0(S)$, the connected component of the identity.
\item
The lifts of transformations in $\Aut(S)$  consist of affine transformations with linear part which is upper triangular,
and diagonal entries $(\e, 1)$ where $\e \in \mu_r$, $r=2,4,6$ according to
$B$ being a general curve, or a harmonic  (Gaussian) elliptic curve, or an equianharmonic (Fermat) elliptic curve.
 \item
 There is an exact sequence 
 $$ 0 \ra  \sK \ra \Aut(S) / \Aut^0(S) \cong \Aut(S) / E \ra \mu_r \ra 1,$$
which is indeed a semidirect product, and 
 where $\sK$ is represented by  transformations of the form
 \begin{eqnarray}
 \label{eq:KK} 
 (z_1,z_2) \mapsto ( z_1 + w_1, z_2  + w_2 z_1 + \s),
 \end{eqnarray}
 with $$w_2 \al_i - w_1 \overline{\al_i } \in \Lambda_E, \ i=3,4.$$
 \item
 The above formula says that  $w = (w_1, w_2)$ is such that 
 $ A w \in \Lambda_E \oplus \Lambda_E$, where the $(2\times 2)$ matrix $$A =\begin{pmatrix}
 -\bar\alpha_3 & \alpha_3\\-\bar\alpha_4 & \alpha_4
 \end{pmatrix}$$ has determinant $- m \be_2 \neq0$.
 
 We obtain that these vectors $w$ are in the  inverse image of the lattice $\Lambda_E \oplus \Lambda_E$.
 
  Observing that  $\Lambda_B$
 is contained in this inverse image, as the set of vectors $(\al, \overline{\al })$,
 with $\al \in \Lambda_B$, and this set  maps onto the subgroup $ \ZZ m \be_2 \oplus  \ZZ m \be_2$,
 we  get a description of $\sK$ as follows:
 \smallskip
 \item
 $\sK \cong (\Lambda_E/ \ZZ m \be_2)^2 \cong \ZZ^2 \oplus (\ZZ/m)^2 \cong Hom(\Lam_B , \ZZ \oplus (\ZZ/m)). $
\end{enumerate}

Hence we derive the following:

\begin{thm}
\label{thm:1st}
For a primary Kodaira surface $S$,  $\Aut^0(S) \cong E$, where $E$ is the elliptic curve
fibre of the Albanese map $ f \colon S \ra B$.

Moreover,   $\Aut(S) / \Aut^0(S) $ maps onto $\mu_r$, where $\mu_r$ is the multiplication group of the 
elliptic curve $B$, with Kernel $\sK  \cong Hom(\Lam_B , \ZZ \oplus (\ZZ/m)) \cong \ZZ^2 \oplus (\ZZ/m)^2$.

$\Aut_{\ZZ}(S) =  \Aut^0(S) $, while  $\Aut_{\QQ}(S) / \Aut^0(S) $ corresponds to  the subgroup of $\sK$  which acts trivially on $H^1(S, \CC)$,
 and is isomorphic  to 
  $Hom (\Lambda_B ,   (\ZZ/m)) \cong (\ZZ/m)^2.$

\end{thm}
\begin{proof}
We first see the action of the transformations 
 from $\sK$,
$$ (z_1,z_2) \mapsto ( z_1 + w_1, z_2 + w_2 z_1),$$
on $\Ga = \pi_1(S),$
 where, for simplicity, we have used the action of $\Aut^0(S)\cong E$ to set $\sigma=0$ in \eqref{eq:KK}.

The inverse transformation is  $ (z_1,z_2) \mapsto ( z_1 - w_1, z_2 - w_2 z_1 + w_1 w_2)$,
hence conjugating $ \ga (z_1,z_2) =   (z_1 + \alpha, z_2 + \bar{\alpha} z_1 + \beta)$
we get $$   (z_1  + \alpha, z_2  + \bar{\al} z_1 + \beta   + \bar{\al}  w_1  - w_2  \al  ),$$
with $ \bar{\al}  w_1  - w_2  \al \in \Lambda_E$.

Hence we get a homomorphism $\Lambda_B \cong \pi_1(S)/ \Lambda_E \ra \Lambda_E$.

The action on $\Ga$ is therefore trivial iff this homomorphism is trivial, that is, if and only if $w_1=w_2=0$.

Also the action on the first homology group $H_1(S, \ZZ)$ yields a homomorphism $\Lambda_B  \ra  \ZZ \oplus (\ZZ/m),$
hence also the action on the first integral homology is never trivial, unless $w_1=w_2=0$.
This proves that $\Aut_{\ZZ}(S) =  \Aut^0(S) $.

The action on $H_1(S, \QQ)$ is trivial provided we get a homomorphism in $Hom (\Lambda_B ,   (\ZZ/m)) \cong (\ZZ/m)^2.$

Now, since $e(S)=0$, $b_2(S) = 4$ hence $H^2(S, \CC) $ has dimension $4$, as explained in 
\cite[Proposition 5.3, p.\ 145]{bpv}.

 In fact, $S$ is diffeomorphic to $\Sigma \times S^1$, where $ \phi  \colon \Sigma \ra B $
is a circle bundle and the product structure is obtained writing $ z_2 = x_1 \be_1 + x_2 \be_2$, 
$ x_1 \in \RR/ \ZZ \cong S^1.$

As shown by Kodaira  \cite[p.\ 787]{kodairastructure1}, 
 $H^1(S, \CC) $ is generated by $dz_1, \overline{dz_1}, dx_1$ and  $dz_1 \wedge \overline{dz_1}$
 is an exact form.
 
 Hence $H^2(S, \CC) $, by the K\"unneth formula, is generated by 
 \begin{eqnarray}
 \label{eq:first_summand}
 \CC  dz_1 \wedge dx_1 \oplus \CC  \overline{dz_1} \wedge dx_1= H^1(B, \CC) \wedge dx_1,
 \end{eqnarray}
 and by $H^2(\Sigma, \CC) $.
 
 For the second summand, it appears,  after tensoring with $\CC$, in the Gysin sequence
 $$ 0 \ra H^1(B, \ZZ) \ra H^1(\Sigma, \ZZ) \ra H^0(B, \ZZ)\ra H^2(B, \ZZ)\ra H^2(\Sigma, \ZZ) \ra H^1(B, \ZZ)\ra0.$$
 Since our  transformations act on $B$ via translations, they act trivially on the
 cohomology groups involving $B$, and preserve the Chern class yielding the coboundary map.
 Hence they act trivially on this second summand $H^2(\Sigma, \CC) $.
 
 For the action on the first summand,  i.e.\ the one given in \eqref{eq:first_summand}, it suffices, to see
 that  they act trivially on $dz_1 \wedge dx_1, \overline{dz_1} \wedge dx_1$:
 this follows directly since they act trivially on $dz_1, \overline{dz_1}, dx_1$, which form a basis of $H^1(S, \QQ)$.
 
 The conclusion is that  $\Aut_{\QQ}(S) / \Aut^0(S) $ is isomorphic to the subgroup of $\sK$
 acting trivially on $H^1(S, \CC)$, which is  isomorphic to   
 $$\Hom (\Lambda_B ,   (\ZZ/m)) \cong (\ZZ/m)^2$$
 as stated.
\end{proof}

We consider next secondary Kodaira surfaces $S'$, which are the quotient of a primary Kodaira surface $S$
by a cyclic finite group $H$ of order $n$ dividing $r$  where, as before, $\Aut(B) = B \rtimes \mu_r$. 

In fact, the cover $ S \ra S'$
is the {\bf canonical cover}, that is, the canonical divisor class $[K_{S'}]$ is a torsion class of order $n$
and $S$  is the cyclic covering, inside the total space of  the line bundle $\hol_{S'}(K_{S'})$,  associated to a nowhere vanishing section
of  $H^0 (\hol_{S'}(n K_{S'}))$. 

Hence any automorphism of $S'$ lifts to $S$.

It is easy to see, cf.\ \cite{cattaneo}, 
that the order $n$  of the  cyclic finite group $H$ divides  $r$, since the elements of $H$ act 
nontrivially on the differential form $ d z_1 \wedge d z_2$.

$H$ is  generated by a transformation of the form 
$$ \ga (z_1, z_2) = (\e z_1 + w_1, z_2 + w_2  z_1 + \s  ), \;\;\;  \ \e \in \mu_n.
$$

We consider the  normalizer $N_H$ of $H$ inside $\Aut(S)$; observe first that $E = \Aut^0(S)$,
 acting on $z_2$ by translations,
and $\ga$ commute,  hence $E = \Aut^0(S)$ descends to $E=  \Aut^0(S')$.

The elements  of the quotient $ \Aut(S)/\Aut^0(S)$ are represented by automorphisms of the form
$$  g (z_1, z_2) = (\e_g z_1 + w_1(g), z_2 + w_2(g)  z_1)  ), \ \e_g \in \mu_r,$$

We shall work modulo $E$ and, in the special case $ n = r$,  we observe that since obviously $\ga$ centralizes $\ga$,
 we just need to look at the elements in $\sK$.

  The conditions of normalizing $\ga$ 
  boil down to one of the two conditions 
  
  \begin{enumerate}
  \item[(i)]
   $ \ga \circ g = g \circ \ga $ or
  \item[(ii)]
   $ \ga \circ g \circ \ga = g $,
  \end{enumerate}
  where (ii) is superfluous if $n=2$.
  
  Let us look at condition (i) for any $g$, setting for convenience $ \de : =  \e_g, \ u_j : =  w_j(g)$:
  $$ 
  \ga \circ g  = ( \e \de z_1 + \e u_1 + w_1, z_2 +   z_1 (u_2 + w_2 \de)+  w_2  u_1 + \s)$$
  should be equal to
  $$ g \circ \ga = (\de \e z_1 + \de w_1 + u_1, z_2 +  z_1 (w_2 +  u_2 \e )  + \s  + u_2 w_1).$$
    Instead condition (ii) cannot be verified if $n \neq 2$, since 
  $$ \ga \circ g \circ \ga = (\e \de \e  z_1 + \dots , \dots), \ \ g  = ( \de z_1 + \dots , \dots).$$
   Hence we just need to look at condition (i), which  holds in $\CC^2$ if and only if 
  $$ ( 1 - \e ) u_1  = ( 1 - \de)  w_1, \ {\rm and } \ ( 1 - \e ) u_2  = ( 1 - \de)  w_2, {\rm and } \ w_2  u_1 = u_2   w_1.$$
  Whereas (i) holds on $S$ if and only if 
  $$ ( 1 - \e ) u_1  - ( 1 - \de)  w_1 \in  \Lam_B , $$
  $$\ ( 1 - \e ) u_2  - ( 1 - \de)  w_2 = \overline
  {( 1 - \e ) u_1  - ( 1 - \de)  w_1 } = ( 1 - \bar\e ) u_1  - ( 1 - \bar\de)  w_1,$$
    and finally  $$\ w_2  u_1 = u_2 w_1 .$$
   To simplify things we look at the case $ g \in \sK$, that is, $\de=1$.
 In this particular case we have the conditions:
  $$ ( 1 - \e ) u_1  \in  \Lam_B , \;\;\; u_2 = \bar{u_1}, \;\;\; w_2  u_1 = u_2 w_1 .$$
The first  equation gives us in general a finite group  isomorphic to $\{ [z_1] \in B | (\e-1) z_1=0\}$, which 
 has cardinality $ N =   4, 3, 2,1$ for $n=2,3, 4,6$ and is respectively isomorphic to 
  $$ T= (\ZZ/2)^2,\;  \ZZ/3, \;  \ZZ/2, \; \{ 1 \}.$$
 If we however intersect with the subgroup $\sK$, we must add the conditions 
 
 \begin{equation}\label{cat}  \overline { u_1   } \al_i - u_1  \bar\al_i \in \Lam_E.
 \end{equation}

 In the case $ n < r$, look then at  the  first two conditions: 
 $$ ( 1 - \e ) u_1  - ( 1 - \de)  w_1 \in  \Lam_B  , \;\; 
 (\de-1) w_2 -  (\e-1) u_2 =  (\bar\de-1) w_1 -  (\bar\e-1) u_1 .$$
  If we view the first equation as a linear equation in $u_1$, this equation is solvable, 
 by the surjectivity of the multiplication by $(\e-1)$ on $B$,
  similarly for  the second equation there is a solution $u_2$:
  hence the number of solutions, for fixed $\de$,
 is either $0$ or equal to the number of solutions for the associated linear homogeneous equation. 

Our group is then generated by the solutions with $\de=1$, i.e., those in $\sK$, and one solution 
 with $\de$ a primitive $r'$-th root of unity, with $r'$ divisible by $n$ and dividing $r$.
\bigskip

Dividing by the group generated by $
\ga$ we get a  semi direct product $ T \rtimes (\ZZ/ (r'/n))$ of respective subgroups $T$
 with a cyclic group of order $r'/n$ dividing $r/n$.
 
  We get, respectively for     $ n = 2,3,4,6$, using \eqref{cat},
 
 \begin{itemize}
 \item
 $ T= (\ZZ/2)^2$ for $n=2$, $m$ even,
 \item
 $ T= \ZZ/3$ for $n=3$, $m$ divisible by $3$,
 \item
  $T=  \ZZ/2, $ for $n=4$, $m$ even. 
 \item
 $ T=  \{ 1 \}$ in the remaining cases.
 \end{itemize}

 In fact, for $n=2$,  \eqref{cat} and the equations for $\sK$, 
 yield, for $\de_j \in \{0,1\}$, 
 $$ \frac{1}{2} (\de_3 \al_3 + \de_4 \al_4 ) \bar\al_i  -  \frac{1}{2} \overline{(\de_3 \al_3 + \de_4 \al_4 )} \al_i  \in \Lam_E $$
 
 $$ {\rm for } \ i=3   \Leftrightarrow   \frac{1}{2} \de_4 m \be_2
 \in \Lam_E, \  {\rm for } \ i=4 \Leftrightarrow 
 \frac{1}{2} \de_3  m \be_2  \in \Lam_E $$
 $$
\Leftrightarrow  \frac{1}{2}  \de_3 m,  \frac{1}{2}  \de_4 m  \in \ZZ.
 $$ 
 
 Similar calculations show that for $n=3$ the element $w_1 = \frac{1}{3} (1 - \e)$ yields a solution iff  $3$ divides $m$, and for
 $n=4$ the element $w_1 = \frac{1}{2} (1 + i)$ yields a solution iff $m$ is even.

 Hence
 
 \begin{thm}
 \label{thm:2nd}
 For a secondary Kodaira surface $S' = S/( \ZZ/n)$, where $S$ is a primary Kodaira surface
 whose   Albanese variety $Alb(S)$ is an elliptic curve $B$ with  $\Aut(B)\cong B \rtimes \mu_r$, 
 we have $ n | r$.
 
 And  there is an elliptic curve 
 $E$ such that  $\Aut(S')^0 \cong E$, and $\Aut(S')/ \Aut(S')^0 $ is a finite group
 of the form 
 $$
 \Aut(S')/ \Aut(S')^0 \cong 
 \begin{cases}
 (\ZZ/2)^2  \rtimes (\ZZ/ (r'/2)) & \text{ for } n=2, \ $m$ \ even, { \ (r'/2) | (r/2)};\\
 (\ZZ/ (r'/2)) & \text{ for } n=2, \ $m$ \ odd, { \ (r'/2) | (r/2)};\\
  \ZZ/3  \rtimes (\ZZ/ (r'/3) ) & \text{ for } n=3, \ $m$ \ divisible \ by \ 3, \ \ { (r'/3) | (r/3)};\\
   (\ZZ/ (r'/3) ) & \text{ for } n=3, \ $m$ \ not \ divisible \ by \ 3, \ { (r'/3) | (r/3)};\\
    \ZZ/2  & \text{ for } n=4,  \text{ and }  m  \text{ even } ;\\
      \{ 1 \} & \text{ for } n=4,  \text{ and }  m  \text{ odd },  \text{ or  for }n=6.
 \end{cases}
 $$
 Here,  either $r'=r$ or $ (r'/n)=1$.

  Moreover, $\Aut(S') = \Aut_{\QQ}(S')$.
 \end{thm}
 
 \begin{proof}
Everything has been proven above,
 except for the validity of the third equation
 and the last assertion:  $\Aut(S') = \Aut_{\QQ}(S')$.
 
 The third equation, $\bar{u_1}w_1 =  w_2  u_1  $ holds clearly for $u_1$ if and only if
 it holds for an integer  multiple of $u_1$, $k u_1$.
 
 Now, in each case, there is $ k \in \{2,3\}$ such that the transformation 
 $$ \hat{g} = (z_1 + k u_1, z_2 + k \bar{u_1} z_1)$$
 belongs to the subgroup $\Lambda_B$ in the canonical semi-direct product 
 $$ \pi_1(S) = \Lambda_E \rtimes \Lambda_B.$$
 
 Now, $\ga$ normalizes $ \pi_1(S) $ and the characteristic semi-direct product.
 Hence it suffices to show that $\ga$ commutes with $\hat{g}$.
 
 But the action of $\ga$ on $\Lambda_B$ by conjugation is given by $\e$,
 and since $k u_1$ is a fixed vector for $\e$, the desired commutation holds true, 
 and the third equation is verified.
 
 Observe moreover that the conditions
 $$ (r'/2) \  \text{ divides } \ (r/2) \in \{1,2,3\}, \;\;\; (r'/3) \ \text{ divides } \ (r/3) \in \{1,2 \}$$
 mean: either $r'=r$ or $ (r'/n)=1$.

 Finally, observe that a Kodaira surface $S$ has $e(S)=0$, and Kodaira proved that for a secondary 
 Kodaira surface $ b_1(S') =1$, hence $ b_2(S')=0$.
 
 Since $ H^1(S', \QQ) = H^1(S, \QQ)^\ga$, in our previous notation, a generator 
 is given by $dx_1$. This generator is left fixed because our transformations act as a translation
 on $z_2$.
 \end{proof}

  \begin{remark}
  Theorem \ref{thm:2nd} leaves open two alternatives in the case $ r =4$ or $r=6$, amounting to
  the existence or non existence of a solution $g$ with $\de$ a primitive  $r$-th root of unity.

  Andrea Cattaneo has pointed out that both alternatives may occur, 
  see Proposition 3.3 and especially Remark 3.6 of \cite{cattaneo}.
    
   \end{remark}

 \begin{remark}
If we consider more generally a minimal surface $S$ of class $(VII_0)$, we have again 
$ b_1(S)=1=  q (S), \ b_2(S)= - K^2_S$.

In the case  $b_2(S)=0$,  $\Aut_{\QQ}(S)$ is then the  kernel 
of the homomorphism $\Aut(S) \ra \{ \pm 1\}$ given by the action on $ H^1(S, \ZZ)$,
hence it has index at most 2.

The surfaces of class $(VII_0)$ and with $b_2(S)=0$ are either the Hopf surfaces
or the Inoue surfaces. 

Kodaira  \cite[p.\ 699]{{kodairastructure2}} proved that if $alg(S)=0,  b_1(S)=1,
b_2(S)=0$ and $S$ contains at least one curve, then $S$ is a Hopf surface.

With the same numerical characters, if $S$ contains no curves, then $S$ is an Inoue surface
(\cite{inoue}, \cite{bogomolov}, \cite{teleman}).

\end{remark}

\section{Elliptic bundles and quasi-bundles}
\label{s:e+qe}

We assume now that $f \colon S \ra B$ is an elliptic quasi-bundle, and that it has a pull back
via $ B' \ra B = B' / G$ which is
an elliptic bundle  $f' \colon S' \ra B'$ where $B'$ has genus $b' \geq 2$.

We can also assume, by replacing possibly $B'$ with a finite covering space, that 
$f' \colon S' \ra B'$ is a principal bundle, with fibre $E$, classified by a cocycle in 
 $ H^1(\sE_{B'}) $. Recall that, representing $ E = \CC/ \Lam$, we have the cohomology exact sequence
 \eqref{pb}
$$ 
H^1(B, \Lam) \ra H^1(B, \hol_B) \ra H^1(\sE_B) \stackrel{c}\ra  H^2(B, \Lam) \ra 0,
$$
 where $\ker(c)$ is the subgroup of principal bundles whose
cocycle $\xi$ is  locally constant. 
 
 We recall now some elementary considerations made in \cite{CFGLS24}.

The  universal covering of $S$ is 
 here $$q : \HH \times \CC \ra S =(\HH\times\CC)/ \Gamma,$$
since the pull back of the bundle to the universal covering $\HH$ of $B'$ is isomorphic a product $\HH\times E$,
$\HH$ being Stein and contractible.

In a similar fashion we have a normal subgroup $ \Ga' < \Ga$ such that $\Ga/ \Ga' \cong G$ and 
$$q' : \HH \times \CC \ra S' =(\HH\times\CC)/ \Gamma' \ra S =(\HH\times\CC)/ \Gamma,$$
and our previous remark shows that $\Ga'$ contains a normal subgroup $\pi_E \cong \Lam$ acting  
trivially on $\HH$ and by translations on $\CC$.

Elementary covering spaces  theory now shows that each self-map $\psi : S \ra S$ lifts to the universal covering

As in  \cite[p.\ 316]{topmethods}, we consider the cyclic group $H$ generated by an automorphism $\psi_0$,
and we consider the group  of lifts of elements of $H$, namely
$$ H' : = \{ \psi' : \HH \times \CC \ra \HH \times \CC\mid  \exists\,  \psi \in H  \text{ such that }  \psi  \circ q = q \circ \psi '\}.$$

There is a short  exact sequence 
$$ 1 \ra \Ga \ra H' \ra H \ra 1.$$
 
There is a homomorphism $$ H \ra {\Aut(\Ga)}
 / \Inn (\Ga)$$ induced by conjugation by a lift $\psi'$ of $\psi$.
The fact that this action is defined only up to inner conjugation amounts to the fact that changing the base point
 $x_0$ to $y_0$  we get an isomorphism of fundamental groups $\pi_1(S, x_0) \cong \pi_1 (S, y_0)$ which is only
defined up to inner conjugation.

In particular, the group $\Aut(S)$ is  the quotient
\begin{equation}\label{univ-cover} \Aut(S) =\widetilde \sN_{\Ga} / \Ga, \end{equation}
where $\widetilde\sN_{\Ga}$ is the normalizer of $\Ga$ inside $\Aut( \HH \times \CC)$.

 Indeed any lift  $\psi'$ of an automorphism $\psi\in \Aut(S)$ yields  a commutative diagram
\[
\begin{tikzcd}
\HH\times\CC\arrow[r, "\psi ' "] \arrow[d, "q"']&  \HH\times\CC \arrow[d, "q"]\\
S \arrow[r, "\psi"]& S
\end{tikzcd}
\]
where $q$ is the quotient map. It follows that $\psi' \Gamma =\Gamma \psi '$, that is, $\psi' \in \widetilde \sN_\Gamma$. Conversely, any $\psi '\in \widetilde \sN_\Gamma$ descends to $S$.

Consider then  the exact sequence 
$$ 1 \ra \Ga \ra  \widetilde \sN_\Gamma \ra \Aut (S) \ra 1,$$
 where  any lift $\psi'$ of $\psi$ defines, via conjugation, an automorphism of $\Ga$ which is well defined up to
inner conjugation. 

Take $x_0$ to be a base point on $S$ and consider the fundamental group
$\Ga = \pi_1(S, x_0)$: then, setting $y_0 : = \psi (x_0)$, we have  $\psi_*  : \pi_1(S, x_0) \ra \pi_1(S, y_0)$,
and choosing a path $\de$ from $x_0$ to $y_0$ we get a fixed isomorphism 
$$  \pi_1(S, y_0) \cong  \pi_1(S, x_0),$$
obtained via conjugation by $\de$ (we take $\de$ to be the trivial path if $y_0= x_0$)
and,  composing $\psi_*$ with this isomorphism, we get an automorphism of $\Ga$.

By changing the lift we can make the automorphism of $\Ga$ induced by conjugation equal to 
$$\psi_* : \pi_1(S, x_0) \ra \pi_1(S, y_0)$$
(see especially \cite[p.\ 316]{topmethods}, where $H'$ is called  the orbifold fundamental group
associated to a properly discontinuous subgroup $H$ of $\Aut(S)$: if the action of $H$ is free, $H'$ 
is the fundamental group of $S/H$,  otherwise, if $x_0$ is a fixed point of $H$, it is the semidirect product
$\Ga \rtimes H$, where conjugation is given by the action of $H$ on $\Ga = \pi_1 (S, x_0)$,
where $\psi \mapsto \psi_*$).
 
This property defines  a  lift $\tilde\psi$ of $\psi$, such that, considering the isomorphism 
 $\psi_*$  indicated above  (depending on the choice of $\de$), 
and setting, for $\ga \in \Ga$,  $\ga' : =  \psi_*(\ga)$,
$$ \tilde\psi : \HH \times \CC \ra \HH \times \CC$$
 enjoys the following important property
\begin{equation}\label{adding-prime} \ga' \circ \tilde\psi = \tilde\psi \circ \ga
\end{equation}
(observe that  both left hand side and right hand side are lifts of $\psi$ which  take the same value on
the same base point $x_0'$ lying above $x_0$).

Write now
 $$ \ \tilde\psi (t,z) = (\tilde\psi_1 (t,z), \tilde\psi_2 (t,z))=  (\tilde\psi_1 (t), \tilde\psi_2 (t,z)),$$
where the last equality follows from Liouville's Theorem,  which also implies that for $\ga \in \Ga$
(which is a lift of the identity of $S$)
$$ \ga (t,z) = (\ga_1 (t), \ga_2 (t,z)).$$

Observe that  condition \eqref{adding-prime} 
 spells out as 
\begin{equation}\label{first-condition}  \tilde\psi_1 (\ga_1(t)) =   \ga'_1 (\tilde\psi_1 (t)) , \end{equation}
\begin{equation}\label{second-condition}   \tilde\psi_2 (\ga_1(t), \ga_2 (t,z))) = \ga'_2 ( \tilde\psi_1 (t),  \tilde\psi_2 (t,z))). \end{equation}

Conversely, any such map $\tilde\psi$ satisfying the above two equations descends to $S$.

Looking at elements $\ga$ in the subgroup $\pi_E < \Ga$, acting trivially on $\HH$ and by translations on $\CC$, we obtain from \eqref{adding-prime},
since  we have $ \ga (t,z) = (t , z + c_{\ga_2} ),$
 $$    \tilde\psi_2 (t , z + c_{\ga_2} ) =    \tilde\psi_2 (t,z) + c_{\ga'_2}. $$
 Hence $\partial \tilde\psi_2 (t,z)  / \partial z$ is $\pi_E$-periodic, thus constant as a function of $z$, and as usual  
 we can write 
 $$ \tilde\psi_2 (t,z) = \la z + \phi(t),$$
 and, for $t$ fixed,   $ \tilde\psi_2$ descends, as already claimed, to an automorphism of $E = \CC / \pi_E$.

Similarly, for a general $\ga$,  $\ga_2 (t,z)$ is an affine function of $z$, $\ga_2 (t,z) = \la_{\ga} z + c_{\ga} (t)$.

 Here $\la$ and $\la_\ga $ are roots of unity, of order dividing $4$ or $6$.

An important remark is now that, if $\psi \in \Aut(S)$, then automatically, by the canonicity of the
covering $B' \ra B = B' /G$, the automorphism $\psi$ lifts to $\psi' \in \Aut(S')$. Then, since 
$ \ \tilde\psi (t,z) =  (\tilde\psi_1 (t), \tilde\psi_2 (t,z)),$ $\psi'$ acts on the base of the covering, $B'$,
which has genus $\geq 2$. If we now assume that $\psi' \in \Aut_{\QQ}(S')$, then by Lefschetz
the action on $B'$ is the identity.

If  $\psi \in \Aut_{\QQ}(S)$, and the genus of $B$ is at least $2$, in particular 
if $f$ is a bundle,  then necessarily $\psi$ acts as the identity on $B$.
We can therefore assume that $\tilde\psi_1 (t,z) = t$, and $ \tilde\psi_2 (t,z) = \la z + \phi(t).$

Otherwise,  if $B$ has genus $1$, there is the possibility that  $\psi$ acts as a translation on $B$,
which must however preserve  the set of points $\sB_m$ whose fibre is multiple with a fixed multiplicity $m$.

Let $ r \in \{2,4,6\}$ such that $\Aut(E) = E \rtimes \mu_r$, and use the notation $\e , \la$ for an element of $\mu_r$.

Then  $$ \tilde\psi_2 (t,z) = \la z + \phi(t), \  \ga_2 (t,z) = \e_{\ga} z + c_{\ga} (t).$$

Observe  that, if we write 
\begin{equation} \label{product-aut}\tilde\psi(t,z) = (\tilde\psi_1 (t) , \la z + \phi (t)), \;\; \ga(t,z) = (\ga_1 (t),  \e_{\ga} z + c_{\ga} (t)),\end{equation}
and set,  for $\ga \in \Ga$,  $\ga' : = \psi_* (\ga)  $,
then   the conditions \eqref{first-condition} and \eqref{second-condition}
read out as:

$$  \ga'_1( \tilde\psi_1 (t)) = \tilde\psi_1 (\ga_1( t)) , \;\;\;  \ga'_2(t,\la z + \phi (t)) =  \la \ga_2(t, z) + \phi (\ga_1( t)) .$$

Observing then  that,  for $\ga \in \Ga$, $ \ga  \mapsto \ga'$ is the effect of conjugating by $\tilde\psi$, we obtain, abbreviating $ 
\ga_2 (t,z) = \e  z + b$,
that 
\begin{equation}\label{conj}
\ga'_2(t,z) = \e z + \la b - \e \phi (\tilde\psi_1^{-1} (t)) + \phi (\ga_1  (\tilde\psi_1^{-1} (t))).
\end{equation}

Hence condition \eqref{second-condition} is equivalent, setting
\begin{equation}\label{Udef}  U(t) : = \phi (\ga_1( t))  -  \e \phi ( t) ,
\end{equation}  to:
\begin{equation}\label{U}  U(t) = U (\tilde\psi_1(t)) .
\end{equation}

This formula is particularly interesting in the  case   where $\e = 1\,\; \forall\, \ga$,
in particular if we have a principal bundle.

\begin{remark}
In the case where $\Aut_{\QQ}(S)$ acts trivially on $B$, as it happens for an elliptic bundle,
then the condition that $\psi$ normalizes $\Ga$ can be simplified, since  $\tilde\psi_1(t) = t$,
 whence $\ga'_1(t) = \ga_1(t)$, and boils down to 
 $$ \ga' _2(z) = \e z + \la b - \e \phi (t) + \phi (\ga_1  (t)) \Leftrightarrow \e_{\ga} = \e_{\ga'} = \e , \ c_{\ga'}(t) = \la c_{\ga}(t) - \e \phi (t) + \phi (\ga_1  (t)).$$

\end{remark}

\section{The K\"ahler non-algebraic quasi-bundles}
\label{s:qb}

If $ S \stackrel f\ra B$ is a K\"ahler non-algebraic quasi-bundle
 we  have, as already seen,  $b \geq 1$, $b' \geq 2$, and the principal bundle $S' \ra B'$ is topologically trivial, see Theorem \ref{bundle}.

There is a finite group $G$ acting  freely on $S', B'$, with  $ S = S' / G$, and we shall take a minimal realization of $S'$ as 
such a quotient, so that  $G$ acts faithfully on the fibres $E$.

It follows that  $\Aut_{\QQ}(S) < \Aut(S)$ 
corresponds,
 modulo $G$, 
to  the subgroup of the normalizer of $G$ in   $\Aut(S')$ which acts as the identity on 
$H^*(S, \QQ)= H^*(S', \QQ)^G$.

\subsection{The case $S=S'$}
Here $S$ is a topologically trivial bundle, $b \geq 2$, $\Aut_{\QQ}(S)$ acts trivially on the base, and it must act
on the fibres $E$ via translations, since $H^1(S, \QQ) = H^1(B, \QQ)\oplus H^1(E, \QQ)$
and 
\begin{eqnarray}
\label{eq:tilde-psi}
\tilde{\psi} (t,z) = ( t, \la z + \phi (t)),
\end{eqnarray}
 hence $\la $ yields the direct summand $End ( H^1(E, \QQ))$  of $End (H^1(S, \QQ))$.

We have then $\la= \e_\ga=1$, and our previous formulae specialize to 
$$ (t,z) \mapsto (t, z + \phi (t));$$
here the  condition to be in the normalizer boils down to  $$ \ c_{\ga'}(t) =  c_{\ga}(t)  + \phi (\ga_1  (t)) -  \phi (t)
 \;\;\; \forall \gamma\in\Gamma.$$

The easy remark is that, if $\ga' = \ga \ \forall \ga\in\Gamma$, then  $\phi (\ga  (t)) =  \phi (t)$ and $\phi$ is a holomorphic function on $\HH/ \Ga$, hence it is constant. 

We want to reach the same conclusion in   the more general case that 
there is a $\ga\in\Gamma$ with $\ga' \neq  \ga $.

The bundle homotopy sequence
$$ 1 \ra \pi_E \ra \pi_1(S)= \Ga  \ra \pi_1(B) \ra 1$$
 is here a split sequence.

Hence $\Ga = \pi_E \times \pi_1(B)$, and we can analyse separately the automorphisms in each factor,
those with
$ \ga(t,z) = (t, z + \ga_2)$,  $\ga_2 \in \pi_E$, which commute with $\tilde{\psi}$,
and those with $ \ga(t,z) = (\ga_1(t), z + c_\ga(t))$. 

Here $\phi (\ga  (t)) -  \phi (t) = c_{\ga'}(t) -  c_{\ga}(t) $ is an element  in $\pi_E$,
thus it is a constant $a_\ga$,  
hence the derivative  $\phi'(t)$ satisfies 
 $\phi'(t) = \phi'(\ga(t)) \frac d{dt}\ga(t)$ and yields 
 a holomorphic section $\om : =  \phi'(t) dt$ of the  canonical bundle of $B$.

Our formulae say that the periods $d_\ga$ of this 1-form $\om$  lie in the subgroup $\pi_E$.

Now, the first homology sequence 
$$ 0 \ra \pi_E \ra H_1(S, \ZZ) \ra H_1(B, \ZZ) \ra 0,$$
which is also split, says that $\psi$ acts trivially on the first and last term of this sequence, 
and the action yields a homomorphism $\psi_* \colon H_1(B, \ZZ) \ra \pi_E$
such that $\psi_* (\ga) = d_\ga$.

We have found that $\phi$ defines a holomorphic map $\Psi \colon B \ra E$, and 
$\om$ is the pull-back of the generator of $H^0( \Omega^1_E)$.

If  the periods of $\om$ are trivial,  then we  can conclude that $\om =0$, hence $\phi$ is constant, 
hence  $ \ c_{\ga'}(t) =  c_{\ga}(t) $,  equivalently $\ga' (t,z) = \ga (t,z)$.

That this is so follows from the fact that the action of $\psi$ on the rational cohomology is trivial,
since we have found that  $\Aut_{\QQ}(S)  $ is isomorphic to the space of holomorphic
maps  $\Psi \colon B \ra E$, which map to zero inside  $ Hom (H_1(B, \ZZ) ,  
H_1(E, \ZZ)).$ Then the dual map $H^ 1(E, \QQ) \ra H^ 1(B, \QQ) $ is also zero,
hence by the Hodge decomposition $\om=0$, hence  $\phi$ is constant.

We have then

\begin{thm}\label{principal}
If $S \ra B$ is a K\"ahler principal bundle with fibre $E$ and with $b \geq 2$, then 
$\Aut_{\QQ}(S) \cong \Aut^0(S) \cong E$.
\end{thm}

\bigskip
\subsection{The case $S= S'/G$}
\label{ss:S'/G}

As before,
$\Aut(S)$ lifts to $S'$, as $N_G / G$, where $N_G$ is the normalizer of $G$ inside  $\Aut(S')$.

We continue to assume that $S'$ is a K\"ahler principal bundle with fibre $E$, and we record 
the following chain of  inclusions 
$$ \Aut(S') \supset \Aut_{B'}(S') \supset Mor(B', E) \supset E,$$
where the second term is the subgroup acting trivially on the base $B'$,
 the  third  term consists of automorphisms of the shape $(t,z) \mapsto (t, z + \phi(t))$
 and the fourth term corresponds to the automorphisms with constant $\phi$.

The first issue is the action of $\Aut_{\QQ}(S)$ on the base $B$: it is trivial if $b \geq 2$,
so that $\Aut_\QQ(S)=\Aut_{B,\QQ}(S)$.
For now, we focus on this case.

   \subsection{The case  with $b \geq 2$}
   \label{ss:b=>2}
   
   Clearly the automorphisms of $\Aut^0(S')$ which descend to $S$ are exactly  those translations 
$$ \tau_w : (x,z) \mapsto (x, z+w)$$ which normalize $G$, i.e., 
    setting $ z \ra L_g z + b_g$ for the action of $G$ on $E$, 
   those automorphisms such that $$ g \circ \tau_w = \tau_w \circ g \Leftrightarrow (gx, L_g z + b_g + w) = (gx, L_g z + b_g + L_g w) \;\;
\ \forall g \in G.
$$
Hence we get the group $E^G : = \{ w \in E | L_g (w)=w \}  \subset  \Aut^0(S') \cong E$
 which momentarily plays a central role in determining $\Aut^0(S)$.

\begin{lemma}
\label{claim:E^G}
$\Aut^0(S) \subset E^G$.
\end{lemma}

\begin{proof}
Consider an automorphisms  $\Psi \in \Aut^0(S)$: 
by our previous discussion it follows that $\Psi$ acts trivially on the base $B$.

Moreover, $\Psi$ takes  the form of \eqref{eq:tilde-psi} with $\lambda=1$,
$$ \tilde\psi: \;\;  (t,z) \mapsto (t, z + \phi (t)),$$
because the derivative of $\tilde{\psi}_2 (t)$ takes discrete values ($\la \in \mu_r$).

The  condition to be in the normalizer boils down to  \begin{equation}\label{constant} \ c_{\ga'}(t) =  c_{\ga}(t)  + \phi (\ga  (t)) -  \phi (t).
\end{equation}
By taking $\ga$ in the subgroup $\Ga'$ such that $S' = ( \HH \times \CC)/\Ga'$,
and since  $\Psi$ is homotopic to the identity, we infer that $\ga= \ga'$, hence $\phi(t)$ is constant,
and the lemma follows from the previous discussion.
\end{proof}
   
    We have a finite index   subgroup $Mor (B, E) < \Aut_{B, \QQ}(S)  $,
of index $ \leq r  = |\Aut(E, 0)|$, consisting of the transformations  for which $\la=1$,
that is, the transformations $ (t,z) \mapsto (t, z + \phi (t)).$ 

 Here $E^G$ is the subgroup of those automorphisms for which $\phi(t)$ is constant.
 
 Equivalently, as we saw,  those automorphisms for which $\phi_* : H_1(B, \QQ) \ra  H_1(E, \QQ)$ is zero.

\begin{cor}
\label{cor:E^G}
$\Aut^0(S)\cong E$ if $G$ acts by translations on $E$, and it is trivial otherwise.
\end{cor}

\begin{proof}
If $G$ does not act by translations on $E$, then $E^G$ is discrete in $E$,
and the claim follows from Lemma \ref{claim:E^G}.
Otherwise, elements in $E$ obviously induce automorphisms in $\Aut^0(S)$ by way of fibrewise translations.
\end{proof}

 \subsection{Excluding the case of an elliptic non-principal bundle with $b \geq 2$}
 \label{ss:not_elliptic}
 
 In this case $G$ acts freely on $B'$, and the monodromy action is not via translations, 
 hence $\Ga$ does not act on $z$ via translations, hence   $\Aut^0(S)$ is trivial 
  by Corollary \ref{cor:E^G}.

  The fundamental group sequence:
 $$1 \ra  \pi_1(E) \ra \pi_1(S) \ra \pi_1(B) \ra 1$$
 is no longer a product, and it yields an exact sequence for the first homology:
 $$ H_1(E, \ZZ) \ra H_1(S, \ZZ) \ra H_1(B, \ZZ) \ra 0.$$
 
 By the K\"ahler assumption the image of 
 the first map has even rank, hence the map is either injective or zero. 
 
 In the first alternative, since conjugation by $\pi_1(B)$ on $\pi_1(E) $ is via the epimorphism
 $\pi_1(B) \twoheadrightarrow G$, follows that 
  $H_1(E, \ZZ) \ra H_1(S, \ZZ)$ factors through the $G$-covariants $H_1(E, \ZZ)_G$, 
 hence the action of $G$
 on   $H_1(E, \ZZ) $ is trivial, a contradiction.
 
Hence the image of  $H_1(E, \ZZ) $ is trivial, but this contradicts Theorem \ref{bundle}, by which $b_1(S) = b_1(B) + 1$
if $c(\xi)\neq 0$.
Thus this case is excluded,  i.e.\ $S$ cannot be K\"ahler.

 \subsection{The case of a quasi-bundle with $b \geq 2$ }
 
 Here the first case to consider is the one where
  $G$ does not act via translations on $E$: 
then $H_1(E, \ZZ)_G$ has rank $1$, 
and by the K\"ahler assumption $H_1(E, \QQ)$
has trivial image in $H_1(S, \QQ)$.

Hence the bilinear map $\Lam^2 (H_1^{orb}(f)) \ra H_1(E, \QQ)$ is nontrivial,
contradicting that $\Lam^2 (H_1(B')) \ra H_1(E, \QQ)$ is trivial.

The conclusion is that $S$ is non-K\"ahler  if $G$ does not act by translations.

\medskip

 To cover the remaining case, recall that we choose $r\in\{2,4,6\}$
such that $\Aut(E)=E\rtimes \mu_r$.

\begin{lemma}
\label{lem:E=translations}
If $S$ is a quasi-bundle with $b \geq 2$,
then $G$ acts by translations on $E$,
and $ E < \Aut_{B, \QQ}(S) = \Aut_\QQ(S)  $
  is a subgroup of  index $\leq r$.
\end{lemma}

 \begin{proof}
  From the previous discussion, 
 we  infer that  $G$ acts by translations on $E$, hence  $\Aut^0(S) \cong E$
by Corollary \ref{cor:E^G}.
 
 We have  the orbifold fundamental group sequence:
 $$ \pi_1(E) \ra \pi_1(S) \ra \pi_1^{orb}(f) \ra 1,$$
 such that for the subgroup $\pi_1(S') $, kernel of the epimorphism to $G$, 
 $$\pi_1(S') = \pi_1(E)  \times \pi_1(B'),$$
 where $ 1 \ra  \pi_1(B') \ra  \pi_1^{orb}(f) \ra G \ra 1$.

 Then,  since $G$ acts by translations, $H_1( E, \ZZ) = H_1( E, \ZZ)_{ \pi_1^{orb}(f) }$;
observe moreover that 
 we have
 $$    H_1( E, \ZZ) \ra  H_1(S, \ZZ) \ra H_1^{orb}(f)  \ra 0,$$
 inducing
 $$    H_1( E, \QQ) \ra  H_1(S, \QQ) \ra H_1(B, \QQ)  \ra 0.$$
 Since $ \phi : B \ra E$ is such that the corresponding $\psi^*$ 
 acts as the identity on $   H_1(S, \QQ)$, it acts as the identity on  $H_1(B, \QQ)$  and we get,
 again by equation \eqref{constant},  that 
  $\phi_* : H_1(B, \QQ) \ra  H_1(E, \QQ)$ is zero, i.e., $\phi$ is constant, implying that $ E < \Aut_{B, \QQ}(S)  $
  is a subgroup of finite index $\leq r$ as stated.
  \end{proof}

\subsection{Proof of Theorem \ref{thm:quasi}}

Having completed the analysis of the $b\geq 2$ case,
we are now ready to prove the third main theorem of this paper.

\begin{thm}
\label{thm:quasi'}
Let $S \ra B$ be a K\"ahler non-algebraic quasi-bundle with fibre $E$, 
which is a quotient 
 $S = S'/G$ of a principal elliptic bundle $ S' \colon \ra B'$, chosen minimally. 
 If $S$ is not an  elliptic  bundle, 
 then,
 letting $ b= \mathrm{genus}(B)$,
\begin{enumerate}
  \item 
   $G$ acts by translations on $E$ and
  $\Aut^0(S) \cong E$.
\item
If $ b \geq 2$, then $\Aut^0(S) \cong E$ is a subgroup of  index $\leq 6$  in the group
  $\Aut_{\QQ}(S)$ and $\Aut_{B, \QQ}(S)  = \Aut_{\QQ}(S)$.
\item
If   $b=1$,
 then $G$ is abelian, 
 and
 $ \Aut_{\ZZ}(S) / \Aut^0(S) $ has order at most $2$.
\item
For $b=1$, 
 $ \Aut_{\QQ}(S) / \Aut^0(S) \cong \sT,$ where  $\sT$ is the maximal group of  translations on $B$
which leave invariant  the monodromy of the covering $B' \ra B = B'/G$  (hence the ones which lift to automorphisms of $B'$).
\item
 In particular, for the class of surfaces with $b=1$, which have $p_g(S)=1, \chi (S)=0, K^2_S=0, $ there is no 
absolute upper bound for $|\Aut_{\QQ}(S) / \Aut^0(S) |$, but there is the  relative upper bound $|\Aut_{\QQ}(S) / \Aut^0(S) | \leq P_2(S).$ 
\end{enumerate}

\end{thm}

\begin{proof}
 For $b\geq 2$,
(1) follows from Corollary \ref{cor:E^G} and Lemma \ref{lem:E=translations},
 and (2)  follows from \S\,\ref{ss:S'/G} and Lemma \ref{lem:E=translations}
 where  the index is bounded by $|\mu_r|$, and $ r \leq 6$,
 determined by $\Aut(E)=E\rtimes\mu_r$ as before.

%

\medskip
 
We now focus on the remaining case $b=1$.

Since $S \ra B$ is not an  elliptic bundle,   $\Aut_\QQ(S)$
  acts on $B$
via a group $T$ of translations,
which must permute the set of critical values (points over which we have a multiple fibre).

Since $T$ is finite, we infer that $\Aut^0(S)\subset \Aut_{B,\QQ}(S)$.

By our previous observations, $b=1$ implies $q(S)=2, p_g(S)=1$, and by relative duality,
$f_* (\Omega^1_{S|B})$ is trivial, hence  $G$ acts via  translations on $E$,  
as stated in (1),
and thus $G$ is abelian
 because the action on $E$ was assumed to be faithful.

 This proves the first statement of (3).

We obtain again a split  exact sequence 
$$ 0 \ra H_1(E, \QQ)  \ra H_1(S, \QQ) \ra H_1(B, \QQ) \ra 0.$$

Moreover, by Lemma 2.4 of \cite{CFGLS24}, the group $\Aut_{\ZZ} (S)$ has a subgroup $H_B$ of index at most $2$
which acts trivially on the base $B$, hence we shall consider    first automorphisms in $\Aut_{B}(S)$ (and $\Aut_{B,\QQ}(S)$).

For these, we take  lifts  in $\Aut_{B'}(S')$, hence with $\psi$ acting via $ (t,z) \mapsto (t, \la z + \phi(t))$,
$\la \in \mu_r$.

Arguing as in the case $S=S'$, we find that $\la=1$ 
 if $\psi\in\Aut_{B,\QQ}(S)$, hence we get again a  morphism $B\ra E$,
and as in the proof of Theorem \ref{principal} we see that  it must  be constant.

We therefore deduce that   
\begin{eqnarray}
\label{eq:Aut_ZZ}
\;\;\;\;
\Aut_{B,\QQ}(S) = \Aut^0(S) \cong E \;\;\; \text{ and } \;\;\;
 |\Aut_{\ZZ}(S) / \Aut^0(S)| \leq 2,
 \end{eqnarray}
in particular completing the proof of (1) and (3).

\smallskip

Returning to the full group $\Aut_\QQ(S)$,
the rational cohomology of $S$, $ H^*(S, \QQ)$, equals $ H^*(S', \QQ)^G$, and since $S'$ is topologically a product,
and $G$ acts on $E$ via translations, we infer that
$$ H^*(S', \QQ)^G =  [H^*(B', \QQ) \otimes  H^*(E, \QQ)]^G = H^*(B', \QQ)^G \otimes  H^*(E, \QQ) = H^*(B, \QQ) \otimes  H^*(E, \QQ) .$$

This shows that $ \Aut_{\QQ}(S)$ consists of  the automorphisms which act on $B$ via translations and which also act
on the fibre $E$ via translations.

We now define the maximal group $\sT$ of allowable translations on $B$: we have the monodromy homomorphism
$ \nu \colon \pi_1 (B^*) \ra G$, and letting $\sB$ be  the (non-empty)  branch locus of $ B' \ra B= B' / G$,
we partition $\sB = \cup \sB_g$ according to the local monodromy $g (p)$ at the point $ p \in \sB$.

We define $\sT$ as the maximal group of translations such that each $\sB_g$ is a union of $\sT$-orbits: whence,
these translations preserve the monodromy $\nu$, and lift to automorphisms $\psi$ of $S$.

\begin{lemma}
\label{lem:sT}
$\sT = \mathrm{Im}(\Aut_\QQ(S)\to\Aut(B)) = T$.
\end{lemma}


\begin{proof}
Observe  that the group $\sT$ is a finite group, and recall 
that the subgroup $\Aut_{B, \QQ}(S)$ of $\Aut_{\QQ}(S)$
acting trivially
on the base $B$ consists of transformations where  $\phi(t)$ is constant 
 by \eqref{eq:Aut_ZZ}.

 In order  for $\psi_1\in\sT$ to
lift to an element $\psi\in\Aut_\QQ(S)$,
it thus suffices
 to decide whether there is $\phi(t)$
such that $({\psi}_1 (t) , z + \phi(t))$ induces an automorphism of $S$, that is, 
$$  \phi(\ga t) - \phi(t) =  \phi ( \ga {\psi}_1 (t) ) - \phi ({\psi}_1 (t)  )\;\;\;\; \ \forall \ga\in\Gamma.$$

   
  We see immediately that this equation has obvious solutions: the ones where $\phi(t) $ is constant.
Hence $\sT = T =   \mathrm{Im}(\Aut_\QQ(S)\to\Aut(B))$.
\end{proof}

Recall that $\Aut_{B,\QQ}(S)=\Aut^0(S)$ by \eqref{eq:Aut_ZZ}.
Since the former group defines the kernel of the homomorphism $\Aut_\QQ(S)\to\Aut(B)$,
Lemma \ref{lem:sT} shows that $\Aut_{\QQ}(S) / \Aut^0(S) \cong \sT$.
This proves (4).

For (5), we shall use that the cardinality of the branch set $\sB$ is divisible by $|\sT|$
 because $\sT$ acts by translations, hence each orbit has the same length.

 We want to apply Kodaira's canonical bundle formula to $S$ with given invariants.
It involves the pull back of a degree zero divisor $L$ on $B$ and the  divisor $\sum_{1\leq i\leq s}  (m_i-1) F'_i$ accounting for the multiple fibres  $F_i=m_iF_i'$ at $p_i\in B$.
Since $ K_B$  is zero,
 and $ P_2(S) = h^0 (2 K_S )$, we find
$$ P_2(S) = h^0 (f^*(2K_B+2 L )+  \sum_{1\leq i\leq s} (m_i + (m_i-2)) F'_i) \geq h^0 (\hol_B( 2L +   \sum_{p_i \in \sB} p_i)) =| \sB|.$$

To conclude, we infer the estimates $|\sT| \leq |\sB| \leq  P_2(S)$ which give the bound of (5) using Lemma \ref{lem:sT} again.
\medskip

To arrange for K\"ahler non-algebraic quasi-bundles $S$ as in (5) with arbitrarily big $\Aut_{\QQ}(S) / \Aut^0(S)$,
it suffices to set-up examples with branching loci $\sB$ admitting larger and larger translation groups $\sT$.

 To this purpose, 
  let elliptic curves $B$ and $E=\mathbb C/\Lambda$ be given and
 choose a simple cyclic branched covering $f : B' \ra B$ with  group $G = \ZZ/n$ branched on a divisor $\sB$
of cardinality $n$ (and then the local monodromies are all equal to $[1]$).

We further require that $\sB$ is the orbit of translation by a point $p \in B$ of order exactly $n$.

Then we have an exact sequence 

$$ 1 \ra  \pi_1(B') \ra  \pi_1^{orb}(f) \ra G \ra 1,$$
and we denote by $\Pi$ the orbifold fundamental group  $\Pi := \pi_1^{orb}(f) $.

Then we want to construct a cocycle $\xi \in H^1(\Pi,  (\Lambda\otimes\QQ) / \Lambda) \cong H^1(\Pi,  (\QQ / \ZZ)^2)$
satisfying the required properties.

 In fact, to any such $\xi$ corresponds an elliptic quasi-bundle $S \ra B$ with fibre $E$, using Remark \ref{construction}.

 Since $\xi$ is locally constant, the pull back  $S' \ra B'$, which is associated to the image $\xi'$ of $\xi$ via the 
injection 
$$ H^1(\Pi, (\Lambda \otimes \QQ)/  \Lambda) \ra H^1 (\pi_1(B'), (\Lambda \otimes \QQ )/  \Lambda) = H^1 (B', (\Lambda \otimes \QQ) /  \Lambda)$$  has even first Betti number,
whence it is K\"ahler.

If $\xi $ is of infinite order inside  $H^1(\Pi, (\Lambda \otimes \QQ) /  \Lambda)$, then also  $\xi'$ is of infinite order, and this principal bundle is not algebraic.

We also want that the action of $G$ on $S'$ is free. This amounts to the condition that the cocycle takes values in $(\Lambda \otimes \QQ) / \Lambda$
of torsion order exactly $n$ at the local generators of the orbifold fundamental group. This can be easily done 
using the exact sequences 
$$ \langle \de_1, \dots \de_n, a, b | \prod_j \de_j [a,b], \de_1^n, \dots, \de_n^n \rangle = \pi_1^{orb}(f) \ra \pi_1(B) \ra 1$$
and
$$0 \ra  \ZZ/n \sum_j  [\de_j]  \ra \oplus_j (\ZZ/n) [\de_j] \ra  H_1^{orb}(f) \ra H_1 (B) \ra 0,$$
and the fact that 
$$H^1(\Pi, (\Lambda \otimes \QQ) /  \Lambda) = H^1(\Pi^{ab}, (\Lambda \otimes \QQ) /  \Lambda) =  H^1(H_1^{orb}(f), (\Lambda \otimes \QQ)/  \Lambda )= \Hom (H_1^{orb}(f),  (\Lambda \otimes \QQ) /  \Lambda)$$ 
hence, setting \begin{equation}\label{tors}
\Tors (f): =    (\oplus_j (\ZZ/n) [\de_j])/\ZZ/n \sum_j  [\de_j] 
\end{equation}

$$0 \ra \Hom(H_1(B), (\Lambda \otimes \QQ) /  \Lambda) \ra \Hom(H_1^{orb}(f), (\Lambda \otimes \QQ) /  \Lambda)\ra \Hom  ( \Tors (f), (\Lambda \otimes \QQ)/  \Lambda) \ra 0.$$

We conclude since we find the desired element inside $\Hom  ( \Tors (f), (\Lambda \otimes \QQ)/  \Lambda)$, mapping each $\de_j$ to an element of torsion order exactly $n$.

\end{proof}

\section{Remarks on K\"ahler properly elliptic bundles and quasi-bundles.}

In this section we assume that $f \colon S \ra B$ is an elliptic quasi-bundle, which is non-K\"ahler,
hence $b_1(S)$ is odd.

There is a Galois covering $ B' \ra B = B'/G$ such the pull-back is an unramified covering $S = S' / G$ and $ S' \ra B'$ is
an elliptic bundle.

$B'$ has genus $b' \geq 2$ if the Kodaira dimension of $S$ is $1$, see Theorem \ref{bundle}
and the arguments of Remark \ref{b'geq2}.

Passing to a further Galois  unramified covering of the base, we may assume that 
$ S' \ra B'$ is a principal bundle, which does not admit a K\"ahler structure, therefore with $c(\xi)\neq 0$, and then $b_1(S') = b_1(B')  +1$.

\subsection{The case of a principal bundle} Here we assume that $f \colon S \ra B$ is an elliptic principal bundle.

Then $\Aut_{\QQ}(S)$ acts trivially on the base curve $B$, which has genus $b \geq 2$.

The bundle homotopy sequence
$$ 1 \ra \pi_E \ra \pi_1(S)= \Ga  \ra \pi_1(B) \ra 1$$
 is here a  non-split sequence,and yields an exact sequence
$$ H_1(E, \QQ) \ra H_1(S, \QQ)  \ra H_1(B, \QQ)  \ra 0,$$
and since $b_1(S) = b_1(B')  +1$ the image of $H_1(E, \QQ) \ra H_1(S, \QQ) $
has rank $1$, therefore we conclude that $\Aut_{\QQ}(S)$ acts on the fibres via translations.

Because, if $\la : H_1(E, \QQ) \ra H_1(E, \QQ)$ leaves a rank 1 subspace  invariant, and acts as the identity 
on the quotient, then it has one eigenvalue equal to $1$, whence $\la =1$.

Again we can find a holomorphic map $\Psi : B \ra E$, but to conclude that it is constant seems harder
in this case.

Also, the case of primary  Kodaira surfaces suggests that $\Aut_{\QQ}(S)$ may be larger than $ E = \Aut^0(S)$
(notice that $\Ga$ acts on the fibres via translations).

\bigskip
\subsection*{Acknowledgements} 

We would like to thank Wenfei Liu for useful discussions  at the onset of this project, and  
especially Andrea Cattaneo for
several email exchanges on the automorphisms of Kodaira surfaces, and for correcting some computational error in a previous version.


\begin{thebibliography}{99}

\bibitem[BP83]{BP83}
Wolf P. Barth and Chris Peters, Automorphisms of Enriques surfaces, Invent. Math. 73 (1983), no. 3, 383--411.


\bibitem[BPV84]{bpv}
Wolf P. Barth, Chris Peters, Antonius Van de Ven,
Compact complex surfaces. 
Ergebnisse der Mathematik und ihrer Grenzgebiete. 3. Folge. Band 4. Berlin etc.: Springer-Verlag. X, 304 p. (1984).

\bibitem[Bog76]{bogomolov}
Fedya Bogomolov, 
Classification of surfaces of class $VII_0$ with $b_2 = 0$, Math. USSR Izv 10, 255--269 (1976).



\bibitem[BH75]{BH} 
Enrico Bombieri, Dale Husem\"oller,
Classification and embeddings of surfaces, in `Algebraic Geometry', Arcata 1974, A.M.S. Proc. Symp. 
Pure Math. 29 (1975), 329-420.

\bibitem[BM77]{bo-mu2} 
Enrico Bombieri, David Mumford, 
Enriques' classification of surfaces in char.  p . II.
Iwanami Shoten Publishers, Tokyo, 1977, pp. 23--42.



\bibitem[BM46]{bm1}
Salomon Bochner and Deane Montgomery, Locally compact groups of differentiable transformations. Ann. of Math. (2) 47 (1946), 639--653.
\bibitem[BM47]{bm2}
Salomon Bochner and Deane Montgomery, Groups on analytic manifolds. Ann. of Math. (2) 48 (1947), 659--669.




\bibitem[BR75]{br}
Dan Burns and Michael Rapoport,  On the Torelli problem for k\"ahlerian K3 surfaces. Ann. Sci. \'Ecole Norm. Sup. (4) 8 (1975), no. 2, 235--273.




\bibitem[Cai04]{Cai04}
Jin-Xing Cai, Automorphisms of a surface of general type acting trivially in cohomology. Tohoku Math. J. (2) 56 (2004), no. 3, 341--355.
\bibitem[Cai09]{Cai09}
Jin-Xing Cai,
Automorphisms of elliptic surfaces, inducing the identity in cohomology, Journal of Algebra 322 (2009) 4228--4246.
\bibitem[CLZ13]{clz}
Jin-Xing Cai, Wenfei Liu, Lei Zhang,  Automorphisms of surfaces of general type with $q\geq 2$ acting trivially in cohomology. Compos. Math. 149 (2013), no. 10, 1667--1684.

\bibitem[CL18]{CL18}
Jin-Xing Cai and Wenfei Liu, Automorphisms of surfaces of general type with $q=1$ acting trivially in cohomology. Ann. Sc. Norm. Super. Pisa Cl. Sci. (5) 18 (2018), no. 4, 1311--1348.




\bibitem[Cat03]{barlotti}
Fabrizio Catanese, Fibred K\"ahler and quasi-projective groups. Special issue dedicated to Adriano Barlotti. Adv. Geom. (2003), suppl., S13--S27.


\bibitem[Cat08]{cime}
Fabrizio Catanese, Differentiable and deformation type of algebraic surfaces, real and symplectic structures. Symplectic 4-manifolds and algebraic surfaces, 55--167, Lecture Notes in Math., 1938, Springer, Berlin, 2008.

\bibitem[Cat85]{gnsaga}
Fabrizio Catanese, Superficie complesse compatte, 
Atti del Convegno Nazionale del GNSAGA del CNR, Torino 4--6 Ottobre 1984,
Valetto, Torino (1985) 1-58 (available also through UMI, Italian Mathematical Union).


\bibitem[Cat15]{topmethods}
Fabrizio Catanese, Topological methods in moduli theory. 
Bull. Math. Sci. 5, No. 3, 287--449 (2015).

\bibitem[CFGLS24]{CFGLS24}
Fabrizio  Catanese, Davide Frapporti, Christian Glei\ss ner, Wenfei Liu, and Matthias Sch\"utt, On the cohomologically trivial automorphisms of  elliptic surfaces  I: $\chi(S)=0$, 
Taiwanese J. Math. 29, No. 6,  volume in honour of Y. Prokhorov,  1209--1260 (2025).
 
 \bibitem[CF26]{CF26}
Fabrizio  Catanese, Davide Frapporti,
 Cohomologically or numerically trivial automorphisms of surfaces of general type,
arXiv:2601.18389.  

\bibitem[CL21]{CatLiu21}
Fabrizio  Catanese and Wenfei Liu, On topologically trivial automorphisms of compact K\"ahler manifolds and algebraic surfaces, Rend. Lincei Mat. Appl. 32 (2021), 181--211.

\bibitem[CLS24]{CLS24}
Fabrizio  Catanese,  Wenfei Liu, and Matthias Sch\"utt, On the numerically and cohomologically trivial automorphisms of elliptic surfaces II: $\chi(S)>0$, 
arXiv:2412.17033, to appear in the journal  `Algebraic Geometry', in 2027.

\bibitem[Cat26]{cattaneo}
Andrea Cattaneo,
A note on the automorphism group of Kodaira surfaces, 
 preprint (2026),
arXiv: 2304.09429v4.









\bibitem[Dolg81]{dolg-cime}
Igor V.\ Dolgachev, Algebraic Surfaces with $q = p_g=0$, in C.I.M.E. III 1977, `Algebraic Surfaces', G. Tomassini, editor, 
97-216, Liguori Editore, Napoli (1981) (reedited as Springer Lect. Notes math. CIME 76, (2010)).


\bibitem[DM22]{DM22}
Igor V.\ Dolgachev, Gebhard Martin, Automorphism groups of rational elliptic and quasi-elliptic surfaces in all characteristics, Adv. in Math. 400 (2022) 108274.

\bibitem[Do83]{donaldson} Simon Kirwan Donaldson,
An application of gauge theory to four-dimensional topology,
Jour. Diff. Geometry 18(2), (1983).


\bibitem[DL23]{DL23}
Jiabin Du and Wenfei Liu, On symplectic automorphisms of elliptic surfaces acting on CH$_0$. Sci. China Math. 66 (2023), 443--456. 



\bibitem[FM94]{FM94}
Robert Friedman and John W. Morgan, Smooth four-manifolds and complex surfaces, Ergebnisse der Mathematik und ihrer Grenzgebiete, 3. Folge, Band 27, Springer-Verlag, 1994.

\bibitem[Fuj78]{Fuj78}
Akira Fujiki,  On automorphism groups of compact K\"ahler manifolds. Invent. Math. 44 (1978), no. 3, 225--258.

\bibitem[Gra74]{Grauert74}
Hans Grauert. Der Satz von Kuranishi f\"ur kompakte komplexe R\"aume. 
Invent.
Math., 25 (1074), 107--142.

\bibitem[Inoue74]{inoue}
Masahisa  Inoue, 
On surfaces of class $VII_0$,  Inventiones math., 24 (1974), 269--310.
 
 
 
 \bibitem[Klei80]{kleiman}
 Steven L. Kleiman,  
 Relative duality for quasi-coherent sheaves,
Compositio Mathematica, Volume 41 (1980) no. 1,  39--60.



\bibitem[Kod60]{kodaira1}
Kunihiko Kodaira, On compact analytic surfaces. I. 
Ann. Math. (2) 71  (1960) 111--152.

\bibitem[Kod63]{kodaira2-3}
Kunihiko Kodaira, On compact analytic surfaces. II-III. 
Ann. Math. (2) 77  (1963) 563--626, 78 (1963) 1--40.

\bibitem[Kod63-4]{kodstruct-1-2}
Kunihiko Kodaira, On the structure of compact complex analytic surfaces I-II,
Proc. Nat. Acad. Sci.U.S.A., 50 (1963), 218-221, and 51 (1964) , 1100-1104.

\bibitem[Kod64]{kodairastructure1}
Kunihiko Kodaira, On the structure of compact complex analytic surfaces. I. 
Am. J. Math.  86  (1964) 751--798.

\bibitem[Kod66]{kodairastructure2}
Kunihiko Kodaira, On the structure of compact complex analytic surfaces. II. 
Am. J. Math.  88  (1966) 682--721.

\bibitem[Kod68-3]{kodairastructure3}
Kunihiko Kodaira, On the structure of compact complex analytic surfaces. III. 
Am. J. Math.  90  (1968) 55--83.

\bibitem[Kod68-4]{kodairastructure4}
Kunihiko Kodaira, On the structure of compact complex analytic surfaces. IV. 
Am. J. Math.  86  (1968) 1048-1066.


\bibitem[Kon86]{Kon86} 
Shigeyuki Kond\=o, Enriques surfaces with finite automorphism groups, 
Japan. J. Math. (N.S.) 12 (1986), 191--282. 

 
 \bibitem[Kub76]{Kubert}
 Daniel S. Kubert,
 {Universal bounds on the torsion of elliptic curves},
 Proc. London Math. Soc. {33} (1976), 193--237.
 
 
\bibitem[Kur62]{Kuranishi62}
Masatake Kuranishi, On the locally complete families of complex analytic structures.
Ann.\ Math.\ 75 (1962), 536--577, 1962.

\bibitem[Kur64]{Kuranishi64}
Masatake Kuranishi, New proof for the existence of locally complete families of
complex structures. 
Proc. Conf. Complex Analysis, Minneapolis (1964), 142--154,
1965.

 

\bibitem[Lie78]{Lie78}
 David I. Lieberman,
Compactness of the Chow scheme: applications to automorphisms and deformations of K\"ahler manifolds. Fonctions de plusieurs variables complexes, III (S\'em. Francois Norguet, 1975--1977), pp. 140--186, 
Lecture Notes in Math., 670, Springer, Berlin, 1978. 









\bibitem[Mir89]{Mir89}\
Rick Miranda, The basic theory of elliptic surfaces, Notes of Lectures for Dottorato di Ricerca in Matematica, Universit\'a di Pisa, ETS Editrice Pisa (1989) I-III, 1-106.






\bibitem[Miya74]{miyaoka}
Yoichi Miyaoka, 
K\"ahler metrics on elliptic surfaces,
Proc. Japan Acad, Ser. A 50 (1974) 533.


\bibitem[Muk10]{Muk10}
Shigeru Mukai, Numerically trivial involutions of Kummer type of an Enriques surface, 
Kyoto J. Math. 50 (2010), 889--902. [Addendum available on Mukai's homepage.]
\bibitem[MN84]{MN84}
Shigeru Mukai and Yukihiko Namikawa, Automorphisms of Enriques surfaces which act trivially on the cohomology groups. Invent. Math. 77 (1984), no. 3, 383--397.

\bibitem[Nam74]{Namba}
M. Namba, 
Automorphism groups of Hopf surfaces, Tohoku Math. J. 26 (1974), 133--157.







\bibitem[Pet79]{Pe79}
Chris A. M. Peters,  Holomorphic automorphisms of compact  K\"ahler surfaces and their induced actions in cohomology. Invent. Math. 52 (1979), no. 2, 143--148.


\bibitem[Pe80]{Pe80}
Chris A. M. Peters, On automorphisms of compact K\"ahler surfaces, Jour\'ees de G\'eom\'etrie Alg\'ebrique d'Angers, Juillet 1979/Algebraic Geometry, Angers, 1979, 249--267, Sijthoff\,\&\,Noordhoff, Alphen aan den Rijn-Germantown, Md., 1980.

\bibitem[PS72]{PS72}
Ilia I. Pyatetskij-Shapiro,  Igor Rotislav Shafarevich, 
A Torelli theorem for algebraic surfaces of type K3. 
Math. USSR, Izv. 5(1971), 547--588 (1972).


\bibitem[SS19]{SS19} Matthias Sch\"utt and Tetsuji Shioda, Mordell--Weil Lattices, Ergebnisse der Mathematik und ihrer Grenzgebiete. 3. Folge, Vol 70, 2019.

\bibitem[Ser91]{Ser91} Fernando Serrano, The Picard group of a quasi-bundle, Manuscripta Math. 73 (1991), 63--82.



\bibitem[Ser96]{serrano}
 Fernando Serrano, Isotrivial fibred surfaces, Annali di Matematica pura ed applicata (IC), Vol CLXXI (1996), 63--81.


\bibitem[Shaf65]{steklov}
Igor Rotislav Shafarevich, et al: I. R.; Averbukh, B. G.; Vinberg, J. R.; Zhizhchenko, A. B.; Manin, Yu. I.; Moishezon, B. G.; Tyurina, G. N.; Tyurin, A. N.
Algebraic surfaces. 
Tr. Mat. Inst. Steklova 75, 215 p. (1965).




\bibitem[SZ01]{SZ01}
Ichiro Shimada and
De-Qi Zhang,  
Classification of extremal elliptic K3 surfaces and fundamental groups of open K3 surfaces. 
Nagoya Math. J. 161 (2001), 23--54. 



\bibitem[Shi72]{Shi72}
Tetsuji Shioda, On elliptic modular surfaces, J. Math. Soc. Japan
{24} (1972), 20--59. 


 \bibitem[Siu83]{siu}
 Yum Tong Siu,
 Every K3 surface is K\"ahler,
 Inv. Math.  73 (1983), 139-150.
 
 \bibitem[Tel94]{teleman}
 Andrei Teleman, 
 Projectively flat surfaces and Bogomolov's theorem on class $VII_0$-surfaces, Int. J. Math., Vol. 5, No 2, 253--264 (1994).
 
\end{thebibliography}
\end{document}